\documentclass{article}

\usepackage{amsmath, amsthm, amssymb, tikz}
\newtheorem{thm}{Theorem}[section]
\newtheorem{cnj}[thm]{Conjecture}
\newtheorem{lem}[thm]{Lemma}
\newtheorem{cor}[thm]{Corollary}
\newtheorem{prop}[thm]{Proposition}

\theoremstyle{definition}
\newtheorem{dfn}[thm]{Definition}
\newtheorem{prb}[thm]{Problem}
\theoremstyle{remark}

\newcommand{\R}{\mathbb{R}}
\newcommand{\N}{\mathbb{N}}

\newcommand{\Q}{\mathbb{Q}}

\newcommand{\baire}{\mathcal{N}}

\newcommand{\bb}[1]{\mathbb{{#1}}}
\newcommand{\inv}{^{-1}}
\newcommand{\fr}{^\smallfrown} 
\newcommand{\ip}[1]{\langle {#1} \rangle} 
\newcommand{\pf}[1]{\langle\langle{#1}\rangle\rangle} 
\newcommand{\sm}{\smallsetminus}
 
\newcommand{\bp}[2]{\mathbf{\Pi}_{#2}^{#1}} 
\newcommand{\bs}[2]{\mathbf{\Sigma}_{#2}^{#1}} 

\newcommand{\dom}{\operatorname{dom}}

\newcommand{\res}{\upharpoonright}

\newcommand{\horn}{\mathrm{HornSAT}}
\newcommand{\nae}{\mathrm{NAE}}
\newcommand{\maj}{\operatorname{maj}}
\newcommand{\sat}{\mathrm{SAT}}

\newcommand{\lra}{\Leftrightarrow}

\newcommand{\s}[1]{\mathcal{{#1}}} 
\newcommand{\csp}{\mathrm{CSP}}
\newcommand{\aut}{\operatorname{Aut}}

\newcommand{\pol}{\operatorname{Pol}}

\newcommand{\var}{\operatorname{Var}}

\newcommand{\HS}{\operatorname{HS}}
\newcommand{\HSP}{\operatorname{HSP}}
\title{An algebraic approach to Borel CSPs}
\author{Riley Thornton}

\begin{document}

\maketitle

\begin{abstract}
    We adapt tools from the algebraic approach to constraint satisfaction problems to answer descriptive set theoretic questions about Borel CSPs. We show that if a structure $\s D$ does not have a Taylor polymorphism, then the corresponding Borel CSP is $\bs12$-complete. In particular, by the CSP Dichotomy Theorem \cite{BulatovCSPDichotomy, ZhukCSPDichotomy}, if $\csp(\s D)$ is NP-complete, then the Borel version, $\csp_B(\s D)$, is $\bs12$-complete (assuming P$\not=$NP). We also have partial converses, such as a descriptive analogue of the Hell--Ne\v set\v ril theorem characterizing $\bs12$-complete graph homomorphism problems. We show that the structures where every solvable Borel instance of their CSP has a Borel solution are exactly the width 1 structures. And, we prove a handful of results bounding the projective complexity of certain bounded width structures.
\end{abstract}

\section{Introduction}\label{section:intro}
This paper regards Borel versions of the following general problem: for a fixed finite relational structure $\s D$, when does a structure $\s X$ in the same language admit a homomorphism into $\s D$? We call $\s D$ the template and $\s X$ an instance of $\s D$, and we sometimes refer to a homomorphism as a solution to $\s X$. For example:
\begin{enumerate}
    \item If $\s D=K_n$, then an instance of $\s D$ is a directed graph and a solution is a (vertex) $n$-coloring.
    \item If $\s D$ has domain $\{0,1\}$ and relations
    $\left\{(x_1,...,x_k): \bigvee_{i<j} \neg x_i\vee \bigvee_{i\geq j} x_i\right\}$ then an instance of $\s D$ is an instance of $k\sat$ and a solution is a satisfying assignment.
    \item If $\s D$ is a finite field $\bb F$ equipped with one relation for each affine subspace, then an instance of $\s D$ is a system of linear equations, and a homomorphism in $\s D$ is a solution to the system.
\end{enumerate} Following example (2) above, computer scientists refer to these as constraint satisfaction problems (CSPs). We adopt this convention: 

\begin{dfn}\label{dfn:CSP}
For $\s D$ a finite relational structure, $\csp(\s D)$ is the set of finite structures which admit a homomorphism into $\s D$, and $\csp_B(\s D)$ is the set of codes for Borel structures which admit Borel homomorphisms into $\s D$. 
\end{dfn}

To be precise, we should specify presentations and codings for our structures. For finite structures, the details will not be important for this paper, but see \cite{FederVardi}. For Borel structures, any standard coding should work. Appendix \ref{appendix:codings} gives some details of a convenient coding.

Some of the typical first questions we ask about a Borel combinatorial problem are: Is a classical solution enough to guarantee a Borel solution? Is there some kind of dichotomy theorem for this problem? Is a Borel solution equivalent to a (lightface, effective) $\Delta^1_1$ solution for $\Delta^1_1$ instances? As an example consider the problem of countably coloring graphs. The $G_0$ graph of Kechris, Solecki, and Todorcevic is an is an instance of the countable coloring problem with a solution but no Borel solution. The $G_0$ dichotomy says that $G_0$ admits a homomorphism into any Borel graph with no Borel countable coloring. And, the proof of the $G_0$ dichotomy implies that any $\Delta^1_1$ graph with a Borel countable coloring has a $\Delta^1_1$ countable coloring \cite{KST}. Note that a positive answer to any of our first pass questions for a problem implies an upper bound on the projective complexity of the set of codes of solvable instances. So, a $\bs12$-completeness result rules out all of these niceties. See, for instance, \cite{effectivization} and \cite{TV}.

In computer science, similar first pass questions about a problem include: Is this problem in P? Can it be solved by linear relaxation or by constraint propagation? Is there a finite list of minimal instances with no solution? Remarkably, these problems have all been solved for CSPs. 

Indeed, the class of CSPs was isolated by Feder and Vardi in the 1990s as a rich class of problems where complexity questions could be settled. Despite Ladner's theorem, which says that (assuming P$\not=$NP) there must be many intermediate classes between P and the NP-complete\cite{Ladner}, most natural problems seem to fall into one of the two extremes. As a partial explanation of this phenomena Feder and Vardi conjectured that all problems of the form $\csp(\s D)$ must be either in P or NP-complete \cite{FederVardi}. 

Not too long after Feder and Vardi made their conjecture, Jeavons showed that this conjectured dichotomy must come down to a question about polymorphism algebras \cite{Jeavons}.
\begin{dfn}\label{dfn:polymorphism intro}
A \textbf{polymorphism} of $\s D$ is a homomorphism from the product $\s D^n$ to $\s D$.
\end{dfn} So a polymorphism takes in $n$ solutions to an instance of $\s D$ and returns another solution. Several examples are given after Definition \ref{dfn: polymorphisms}. A classical theorem of universal algebra says that polymorphism algebras ordered by containment are in Galois correspondence with structures under a notion of simulation called pp definability. And, Jeavons showed that pp definitions yield polynomial time reductions. (This has since been generalized greatly, see Theorem \ref{thm:correspondences}).

About a decade later, Bulatov, Jeavons, and Krokhin conjectured an algebraic dividing for polynomial time solvability\footnote{The form of the CSP dichotomy theorem given here is somewhat anachronistic. The original conjecture was in terms of Taylor operations. See the comments after Theorem \ref{thm:intractable equivalents}} \cite{BulatovJeavonsKrokhin}. After decades of work in computer science, combinatorics, and universal algebra, the conjectures of Feder--Vardi and Bulatov--Jeavons--Krokhin were confirmed independently by Bulatov and Zhuk in 2017:

\begin{thm}[CSP Dichotomy Theorem, \cite{BulatovCSPDichotomy}\cite{ZhukCSPDichotomy}] \label{thm:CSP dichotomy theorem}
For a finite relational structure $\s D$, $\csp(\s D)$ is polynomial time solvable if there is a polymorphism $f$ of $\s D$ satisfying
\[(\forall a,e,r)\; f(r,a,r,e)=f(a,r,e,a).\] And $\csp(\s D)$ is NP-complete otherwise.
\end{thm}
In this paper, we give partial algebraic answers to some of our basic questions from Borel combinatorics.

\subsection{Results and conjectures}

Here we lay out the main results of the paper and some related open problems. To state these results we start with a number of definitions.

\begin{dfn}\label{dfn: ess class eff}
Say that $\csp_B(\s D)$ is \textbf{essentially classical} if any time a Borel instance of $\s D$ admits a solution it admits a Borel solution.

Say that $\csp_B(\s D)$ is \textbf{effectivizable} if anytime a $\Delta^1_1$ instance admits a Borel solution it admits a $\Delta^1_1$ solution.
\end{dfn}

These properties say that classical and effective tools respectively are sufficient to understand the Borel problem. Effectivizability plays a surprisingly important role in Borel combinatorics. Non-trivial upper bounds on the complexity of Borel problems almost always come with an effectivization result, and in particular dichotomy theorems in Borel combinatorics seem to always admit a proof via Gandy--Harrington forcing (see for instance \cite{effectivization}).

In the finitary setting, similar classes of problems are closed downward under so-called pp constructions (see Definition \ref{dfn: pp defns}). Since pp construction can be characterized by the polymorphism algebras of $\s D$ and $\s E$ (see Theorem \ref{thm:correspondences}), it follows abstractly that there is some algebraic characterization of these finitary classes. The same is nearly true in the Borel setting, but for technical reasons we need to assume equality is part of our structures.

\begin{thm}[See Corollary \ref{cor:equality closure}]\label{thm:equality closure intro}
Suppose $\s D$ is a structure which includes equality as a relation and $\s E$ is pp constructible in $\s D$. If $\csp_B(\s D)$ is $\bp11$, effectivizable, or essentially classical then so too is $\csp_B(\s E)$.
\end{thm}

The case of dichotomy theorems is somewhat mysterious. See the comments before Theorem \ref{thm: finite basis}. It remains to make the algebraic characterizations of these classes explicit and to test how far these results extend beyond structures with equality.

\begin{dfn}\label{dfn:basic identities}
Say that an operation $f: D^n\rightarrow D$ is
\begin{enumerate}
    \item \textbf{totally symmetric} if $f(x_1,...,x_n)$ only depends on $\{x_1,...,x_n\}$
    \item \textbf{Siggers} if $n=4$ and $f(r,a,r,e)=f(a,r,e,a)$
    \item A \textbf{dual discriminator} if $n=3$ and $f(x,y,z)$ is the repeated values among $x,y,$ and $z$ if there is on and $x$ otherwise.
\end{enumerate}
\end{dfn}

Our first results says that the intractable fork of the CSP dichotomy adapts to the Borel setting. 

\begin{thm}[See Theorem \ref{thm:intractable}]\label{thm:intractable intro}
If $\s D$ does not admit a Siggers polymorphism, then $\csp_B(\s D)$ is $\bs12$-complete.
\end{thm}
\begin{cor}[P$\not=$NP] \label{cor: p np intractable intro}
If $\csp(\s D)$ is NP-complete, then $\csp_B(\s D)$ is $\bs12$-complete
\end{cor}

This theorem gives several interesting new examples of $\bs12$-complete problems in Borel combinatorics. We can also use this theorem to import other finitary complexity dichotomy wholesale. For instance, Corollary \ref{thm:hellnesetril} gives a descriptive set theoretic analog of the Hell--Ne\v set\v ril theorem on graphs.

It is natural to ask if the converse of Theorem \ref{thm:intractable intro} holds:

\begin{prb}\label{prb: intractable converse}
Is there some $\s D$ with $\csp_B(\s D)\in$P but $\csp_B(\s D)$ $\bs12$-complete?
\end{prb}

More generally, we can ask if a Borel CSP complexity dichotomy holds. The answer is yes if we assume $\bs12$-determinacy \cite[Remark 3.3]{TV}, but it is open if this holds in ZFC:

\begin{cnj} \label{cnj: dichotomy}
For every $\s D$, $\csp_B(\s D)$ either $\bp11$ or $\bs12$-complete.
\end{cnj}

Our second main result gives a complete characterization of essentially classical structures. This builds on the characterization of width 1 structures by Dalmau and Pearson \ref{thm:width1 characterization}.

\begin{thm} \label{thm: classical intro}
For any finite relational structure $\s D$, $\csp_B(\s D)$ is essentially classical if and only if $\s D$ admits totally symmetric polymorphism of arbitrarily high arity.
\end{thm}
One direction is Theorem \ref{thm:width1 implies classical} and the other is Theorem  \ref{thm: classical implies width1}. Totally symmetric polymorphisms are quite strong, so the main content of this theorem is to rule out exotic essentially classical problems. Roughly, the only method to prove a problem is essentially classical is the reflection theorem. By examining the proof, we can get the following:
\begin{cor} \label{cor:classical implies effective}
If $\csp_B(\s D)$ is essentially classical, it is effectivizable.
\end{cor}

For effectivizable CSPs in general, we have the following partial result.

\begin{thm}[See Theorem \ref{thm: dd effective}] \label{thm:dd effective intro}
If $\s D$ has a dual discriminator polymorphism, then $\csp_B(\s D)$ is effectivizable.
\end{thm}

This theorem along with Theorem \ref{thm:intractable intro} gives a descriptive analog of the Barto--Kozik--Niven theorem on smooth digraphs, Corollary \ref{cor: smooth digraphs}. The structures indicated in the above theorem are the simplest from the class of bounded width structures (see definition \ref{dfn:cycle consistent}). The following question is natural:

\begin{prb} \label{prb: bddwidth}
Is every bounded width Borel CSP effectivizable?
\end{prb}

These results are almost enough to recover a descriptive set theoretic analog of Schaefer's dichotomy for Boolean CSPs \cite{Schaefer}.

\begin{dfn}\label{dfn: basic structures}
For $k\in \N$, $k$SAT is the structure on $\{0,1\}$ equipped with each of the following relations: for $0\leq i\leq k$
\[D_i(x_1,...,x_k):\lra \left(\bigvee_{j\leq i} \neg x_j\right)\vee \left(\bigvee_{j>i} x_{j}\right)=1\]
For $\bb F$ a finite field, we write $\bb F(k)$ with the structure whose domain is the same as $\bb F$ equpped with one relation for each affine subset of $\bb F^k$, i.e. each set of the form
\[\left\{ (x_1,...,x_n): \sum_i a_ix_i=a_{n+1} \right\}\] for $a_1,...,a_{n+1}\in \bb F$.
\end{dfn}

Note that an instance of $k$SAT is a Boolean formula in $k$CNF. This is abusing notation slightly as computer scientists typically refer to $\csp(k\sat)$ as $k\sat$, but hopefully this will not cause confusion. And, it turns out that all of the structures $\bb F(k)$ are equivalent for all $k\geq 3$ (see the comments after Definition \ref{dfn: pp defns}).

\begin{cor}[See Corollary \ref{cor: schaefer}]\label{cor: schaefer intro}
If $\s D$ is a structure on $\{0,1\}$, one of the following holds:
\begin{enumerate}
    \item $\s D$ has a totally symmetric polymorphism and $\csp(\s D)$ is essentially classical and effectivizable
    \item $\s D$ is pp constructible in 2SAT and $\csp_B(\s D)$ is effectivizable
    \item $\csp(\s D)$ is NP-complete and $\csp_B(\s D)$ is $\bs12$-complete, or
    \item $\s D$ is pp constructible in $\bb F_2(3)$ and vice versa.
\end{enumerate}
\end{cor}

This leaves us with the burning question of projective complexity of linear algebra:

\begin{prb} \label{prb: lin alg}
Is $\csp_B(\bb F(3))$ $\bs12$-complete for every $\bb F$?
\end{prb}

If the answer to Questions \ref{prb: lin alg} and \ref{prb: bddwidth} are both positive,  then Conjecture \ref{cnj: dichotomy} is true, and further we get the following: 

\begin{cnj}\label{cnj: effective iff effective}
For any finite structure $\s D$, $\csp_B(\s D)$ is $\bp11$ if and only if it is effectivizable.
\end{cnj}

\subsection{Acknowledgements}

Thanks to Zoltan Vidny\'anszky and Jan Greb\'ik for helpful conversations, and thanks to Tyler Arant for explaining the coding scheme given in the appendix. The author was partially supported by NSF grant DMS-1764174.

\subsection{Outline of paper}

In Section \ref{section:background}, we review some background on CSPs. Specifically, we make clear the notion of pp constructibility and its relation with the polymorphism algebra. We also describe a number of classes of finitary CSPs relevant to this paper. In Section \ref{section:simple}, we introduce a weakening of pp constructions we call simple constructions which can be employed in the Borel setting, and we establish some combinatorial lemmas relating pp and simple constructions. In Section \ref{section:intractable}, we prove Theorem \ref{thm:intractable intro} and establish some of its consequences. In Sections \ref{section:classical} and \ref{section:effective} we prove Theorems \ref{thm: classical intro} and \ref{thm:dd effective intro} respectively and record some corollaries. And in Appendix \ref{appendix:codings}, we prove $\bs12$-completeness of Borel edge 3-coloring and give some details on Borel codes.

\subsection{Notation and convention}
A relation on a set $A$ of arity $k$ is a subset of $A^k$. We will write $R(x_1,..,x_k)$ to mean $(x_1,...,x_k)\in R$. We will switch freely between these two notations. In particular a unary predicate is just a subset of $A$. Some relations, like equality, are typically written in infix notation. We will write these in parenthesis to emphasize that we mean the associated set. For instance, $(=)=\{(a,b)\in A^2: a=b\}$.

By a relational structure we mean a tuple $\s D=(D, R_1,R_2,...)$, where $D$ is a set and each $R_i$ is a relation on $D$. Relational structures will be denoted throughout by script letters: $\s D, \s E, \s X$, etc. Unless otherwise stated, the domain of a structure will be denoted by the unscripted letter: $D, E, X$, etc.

We say that a structure is finite to mean that its domain is finite and that it comes equipped with only finitely many non-empty relations. When $\s D$ is a finite relational structure, we say that $\s X$ is an instance of $\s D$ when it has the same signature. That is, if $\s D=(D, R_1,...,R_n)$, then $\s X=(X, S_1,..., S_n)$ where $S_i$ has the same arity as $R_i$. We refer to $S_i$ as the interpretation of $R_i$ in $\s X$ and write $S_i=R_i^{\s X}$. The superscript may be dropped if the context is clear. If $\s X$ is an instance of $\s D$, we say $\s D$ is a template of $\s X$.

We will want to perform a few operations on structures. For $\s D=(D, R_1,...,R_n)$ a structure and $R$ a relation on $D$, we abuse notation and write $\s D\cup\{R\}$ for $(D, R_1,...,R_n, R).$ For $A\subseteq D$ we write $\s D\res A$ for the structure with domain $A$ equipped with relations $R_i\cap A^k$ for each $i$ (where $R_i$ has arity $k$.) And, if $R$ and $S$ are relations both of arity $k$ on $A$ and $B$ respectively, then $R\sqcup S$ is the relation on $A\sqcup B$ defined by
\[R\sqcup S=\{e\in A^k\sqcup B^k:R(e)\mbox{ or }S(e)\} \]

For a set $A$, $\pi^k_i: A^k\rightarrow A$ is the $i^{th}$ projection map
\[\pi^k_i(a_1,...,a_k)=a_i.\]
And for $J=\{j_1,...,j_n\}\subseteq \{1,...,k\}$, $\pi_J^k$ is the projection onto the coordinates in $J$:
\[\pi^k_J(a_1,...,a_k)=(a_{j_1},...,a_{j_n}).\] We will drop the superscript when the context is clear.

Our infinite structures will all have domain $\baire$ (or a subspace of $\baire$). By standard universality arguments, our results extend to any Polish space. Most reasonable codings of Borel sets will work for our arguments. See Appendix \ref{appendix:codings} for some details.

Many results in this paper refer to computational complexity classes. Almost anywhere P (or NP) appears it can replaced with the algebraic notion of tractability (or intractability). See Definition \ref{dfn: tractable} and the following comments. In particular most theorems in this paper are true even if P$=$NP. We will point out any exceptions by putting ``(P$\not=$NP)" before any theorems which require this assumption.

\section{Background on CSPs} \label{section:background}

In this section, we review some of the basic algebraic theory of CSPs. The material we cover here is minimal. For a more detailed survey, see \cite{PolymorphismSurvey}.

\begin{dfn}\label{dfn: polymorphisms}
For two relational structures in the same language, $\s X=(X,\tau)$ and $(\s D,\tau)$, a \textbf{homomorphism} from $\s X$ to $\s D$ is a function $f:X\rightarrow D$ so that
\[(\forall R\in \tau, x_1,...,x_n\in X) \;(x_1,...,x_n)\in R^{\s X}\rightarrow (f(x_1),...,f(x_n))\in R^{\s D}.\]

A \textbf{polymorphism} of $\s X$ is a homomorphism from $\s X^n$ to $\s X$, where we take the so-called categorical product. That is, we interpret a $k$-ary relation $R$ in the product structure as $R^{\s X^n}=(R^{\s X})^n$:
\[R^{\s X^n}(\ip{x_{i,j}: i\leq n, j\leq k}):\lra (\forall i)\; R^{\s X}(x_{i,1},...,x_{i,k}).\]

For any structure $\s D$, $\pol(\s D)$ is the algebra of polymorphisms of $\s D$.

A \textbf{clone} is an algebra equipped with every projection operation and whose collection of operations is closed under compositions. For a set of operations $A$, $\ip{A}$ is the smallest clone whose operations contain $A$. We say $A$ \textbf{generates} $\ip{A}$.
\end{dfn}

The following examples are straightforward to verify:
\begin{enumerate}
    \item For any structure $\s D$, $\pol(\s D)$ is a clone
    
    \item For any finite field $\s F$, $\pol(\bb F(3))$ is the collection of linear functionals $f(x_1,...,x_n)=\sum_i a_i x_i$ with $\sum_i a_i=1$
    
    \item The dual discriminator on $\{0,1\}$, also called the majority function or $\maj$, generates $\pol($2SAT$)$.
    
    \item $\horn$ is satisfaction problem for Horn sentences. More specifically, $\horn$ has domain $\{0,1\}$ and includes the unary predicates $\{0\}$ and $\{1\}$ and each relation of the form 
    \[R(x_1,...,x_n,y)\lra (x_1\wedge ...\wedge x_n)\rightarrow y.\] The binary min, $\wedge$, generates $\pol(\horn)$
    
    \item $\nae$ is the 2-coloring problem for ternary hypergraphs, i.e.
    \[\nae=(\{0,1\}, \{(x_1,x_2,x_3): \neg(x_1=x_2=x_3)\}).\] Negation generates $\pol(\nae)$.
    
    \item For the directed 3-cycle $\s D= (\{r,p,s\}, \{(p,r),(r,s), (s,p)\})$, $f\in \pol(\s D)$ if and only if the cyclic permutation $\pi=(rps)$ is an automorphism of $f$. For instance, $\pol(\s D)$ contains $\pi$, the dual discriminator and the binary rock-paper-scissors operation, $(\star)$:
    \[\begin{array}{c| ccc} \star & r & p & s \\ \hline r & r & p & r \\ p & p & p & s \\ s & r & s & s \end{array}\]
    
\end{enumerate}

Note that $\horn$ is not finite, but it will turn out that it is pp definable in a finite number of its relations. For item (3), note that for any $3\times 2$ matrix over $\bb F_2$, applying the majority function to each column returns one of the rows. So, the majority function preserves all binary relations on $\{0,1\}$.

\begin{dfn}\label{dfn: core and hom equiv}

We say $\s X$ is a \textbf{core} if any homomorphism $f:\s X\rightarrow \s X$ is an automorphism (i.e. bijective).

Two structures $\s D$ and $\s E$ are \textbf{homomorphically equivalent} if there are homomorphisms $f:\s D\rightarrow \s E$ and $g: \s E\rightarrow \s D$.

Any structure $\s D$ is homomorphically equivalent to a (unique up to isomorphism) core, which we refer to as the \textbf{core of $\s D$}.
\end{dfn}

Note that if $\s D$ and $\s E$ are homomorphically equivalent, then $\csp(\s D)=\csp(\s E)$. For example, $K_n$ is a core for all $n$, and the core of any bipartite graph is $K_2$.

\begin{dfn} \label{dfn: pp defns}
If $\s D$ and $\s E$ are two structures with the same domain, we say $\s E$ is \textbf{pp definable} in $\s D$ if every relation $R$ of $\s E$ can be written as a positive primitive formula over $\s D$, i.e.
\[R(x_1,...,x_n)\lra (\exists y_1,...,y_k) \bigwedge_i \alpha_i(x_1,...,x_n, y_1,...,y_k)\] where each $\alpha_i$ either asserts a relation from $\s D$ or an equality.

We say $\s E$ is \textbf{pp interpretable} in $\s D$ if it is interpretable in the model-theoretic sense with positive primitive formula, i.e. $\s E$ is a quotient of a pp definable structure of some power of $\s D$ along a pp definable equivalence relation.

And, we say $\s E$ is \textbf{pp constructible} in $\s D$ if there is a chain of structures $\s D=\s E_1, \s E_2,..., \s E_n=\s E$ so that for each $i$, one of the following holds:
\begin{enumerate}
    \item $\s E_{i+1}$ is pp interpretable in $\s E_i$
    \item $\s E_i$ is homomorphically equivalent to $\s E_{i+1}$
    \item $\s E_i$ is a core and $\s E_{i+1}$ is $\s E_i$ expanded by a singleton unary relation $U(x)\lra x=a$
\end{enumerate}
\end{dfn}

For example, given any $a_1,a_2,a_3,a_4,b\in \bb F$
\[a_1x_1+a_2x_2+a_3x_3+a_4x_4=b\] if and only if
\[(\exists y_1,y_2) \; a_1x_1+a_2x_2=y_1 \wedge a_3x_3+y_1=y_2\wedge a_4x_4+y_2=a.\] So $\bb F(3)$ pp constructs $\bb F(4)$ for any finite field $\bb F$. A similar trick shows that $\bb F(3)$ pp constructs every $\bb F(k)$ and that $\horn$ is pp definable in a finite number of its relations.

Since every Boolean formula can be converted to a formula in 3CNF by introducing new variables as above, 3SAT pp defines every structure on $\{0,1\}$. Further, by considering binary expansions of a general relation, one can show that 3SAT pp interprets every structure. Because $\nae$ has a nontrivial automorphism, it does not pp define 3SAT, but one can check that it pp constructs 3SAT (and thus all structures).

The basic idea of the algebraic approach to CSPs is that pp constructions preserve all complexity-theoretic information about a CSP, and that pp constructibility of structures is captured by the (height 1) varieties of their polymorphism algebras. For completeness, we will explain this correspondence in general, though we will only refer to a handful of well-understood varieties in this work.

\begin{dfn}\label{dfn: varieties}
For an algebra $\bb A$, an \textbf{identity} is a true statement of the form
\[(\forall x_1,...,x_n, y_1,...,y_m) \;\tau(x_1,...,x_n)=\sigma(y_1,...,y_m)\]
where $\sigma$ and $\tau$ are compositions of operations from $\bb A$. An identity is \textbf{height 1} if each of $\sigma$ and $\tau$ contain exactly one function symbol (without repetition). As is somewhat standard convention, we will often leave the universal quantifier implicit when defining identities.

The \textbf{variety} generated by $\bb A$, $\var(\bb A)$, is the class of all algebras (in the same signature) which satisfy the same identities as $\bb A$. The \textbf{h1 variety} generated by $\bb A$, $\var_{h1}(\bb A)$ is the class of algebras which satisfy the same height 1 identities.
\end{dfn}

The identities defining Siggers and totally symmetric polymorphisms are height 1, but associativity and idempotence are not. The variety generated by $\pol(3\sat)$ is the class of all projection algebras. On the other hand, the majority function satisfies \[\maj(x,x,y)=\maj(x,y,x)=\maj(y,x,x)\] so $\pol(2\sat)$ and $\pol(3\sat)$ generate different h1 varieties. 

We will need a more detailed analysis of $\var(\bb A)$.

\begin{dfn}\label{dfn:HSP}
 We say that $\bb B$ is a \textbf{reduct} of $\bb A$ if $\bb B$ has the same domain as $\bb A$ and is equipped with a subset of the operations from $\bb A$. In this case, we write $\bb B\subseteq \bb A$. 

We say that $\bb B$ is a \textbf{subalgebra} of $\bb A$ if the domain of $\bb B$ is a a subset of the domain of $\bb A$, and $\bb B$ comes equipped with all of the restrictions of operations from $\bb A$.

For $C$ a class of algebras, we define: 
\[H(C):=\{f(\bb A):\bb A\in C, f\mbox{ a homomorphism}\}\]
\[S(C):=\{\bb B: (\exists \bb A\in C)\; \bb B\mbox{ a subalgebra of }\bb A\}\]
\[P(C):=\{\bb A^\kappa: \bb A\in C, \kappa\mbox{ is a cardinal}\}\]
We write $H(\bb A)$ for $H(\{\bb A\})$ and likewise for $S(\bb A)$ and $P(\bb A)$.
\end{dfn}

The operations $H$ and $S$ will be important for us later. 

\begin{thm}\label{thm:correspondences}
For finite relational structures $\s D$ and $\s E$
\begin{enumerate}
    \item $\s D$ pp defines $\s E$ iff $\pol(\s E)\subseteq \pol(\s D)$
    \item $\s D$ pp interprets $\s E$ iff a reduct of $\pol(\s E)$ is in $\var(\pol(\s D))=\HSP(\pol(\s D))$ iff every identity satisfied by elements of $\pol(\s D)$ is satisfied by elements of $\pol(\s E)$
    \item $\s D$ pp constructs $\s E$ iff a reduct of $\pol(\s E)$ is in $\var_{h1}(\pol(\s D))$ iff every height 1 identity satisfied by elements of $\pol(\s D)$ is satisfied by elements of $\pol(\s E)$.
\end{enumerate}
\end{thm}
\begin{proof}
Statement (1) was proven independently in 60s by Geiger and  Bodnarchuk, Kaluzhnin, Kotov, and Romov \cite{GaloisOG1}\cite{GaloisOG2}. The fact that $\var(\bb A)=\HSP(\bb A)$ is Birkhoff's HSP theorem \cite{Birkhoff}. The equivalence to pp interpetation is essentially by definition and is implicit in Bulatov, Jeavons, and Krokhin \cite{BulatovJeavonsKrokhin}. Statement (3) is due to Barto, Kozik, and Pinsker \cite{reflections}.
\end{proof}

In fact slightly more is true of item (2). For $\bb A=\pol(\s D)$, the polymorphism algebras in $H(\bb A)$, $S(\bb A)$, and $P(\bb A)$ correspond to pp definable quotients, pp definable substructures, and powers of $\s D$ respectively. 

\begin{cor}
For structures $\s D$ and $\s E$, $\s E$ is a pp definable quotient of a pp definable substructure of $\s D$ if and only if $\pol(\s E) \in \HS(\pol(\s D))$.
\end{cor}
\begin{proof}
We will show that any $\s E$ with $\pol(\s E)\in \HS(\pol(\s D))$ is a quotient of a substructure. The converse is straightforward.

Fist note that a subalgebra of $\pol(\s D)$ is a subset of $D$ which is closed under the operations in $\pol(\s D)$. By Theorem \ref{thm:correspondences}, this is the same as a pp definable subset of $D$. 

Now suppose that $\phi:D\rightarrow E$ is a homomorphism of algebras from $\pol(\s D)$ onto $\pol(\s E)$, and extend $\phi$ to a homomorphism $\phi:D^n\rightarrow E^n$. For any operation $f\in \pol(\s D)$, if $\phi(x_1,...,x_n)=\phi(y_1,...,y_n)$ then
\begin{align*} \phi(f(x_1,...,x_n))& =f(\phi(x_1,...,x_n)) \\ \ & =f(\phi(y_1,...,y_n))\\
& =\phi(f(y_1,...,y_n))\end{align*} So the kernel of $\phi$, i.e. the equivalence relation $\phi(x)=\phi(y)$, is invariant under $\pol(\s D)$ and is pp definable in $\s D$. Similarly, the pullback of any relation in $\s E$ is pp definable.
\end{proof}

As a corollary of Theorem \ref{thm:correspondences}, any class of structures which is closed under pp constructions (or interpretations or definitions) admits an algebraic description in terms of height 1 identities (or identities or polymorphisms). We survey several such classes and their algebraic descriptions below.

Bulatov, Jeavons, and Krokhin showed that if $\s D$ pp constructs $\s E$ then $\csp(\s E)$ is polynomial time reducible to $\csp(\s D)$ \cite{BulatovJeavonsKrokhin}. The CSP dichotomy theorem gives the corresponding algebraic characterization of the polynomial time complexity classes for CSPs. We collect a few equivalent characterizations here.

\begin{dfn}\label{dfn: tractable}
If $\s D$ admits a Siggers term, we say $\s D$ is tractable. We say $\s D$ is intractable otherwise.
\end{dfn}

So the CSP dichotomy theorem says that, if P$\not=$NP, then intractable is synonymous with NP-complete and tractable is synonymous with polynomial time solvable.

\begin{thm}\label{thm:intractable equivalents}
$\s D$ is tractable iff any of the following hold:
\begin{enumerate}
    \item $\pol(\s D)$ contains a Taylor operation, i.e. an operation $T$ so that for every $i$ there are $z_1,...,z_n\in \{x,y\}$ so that $T$ satisfies an identity of the form
    \[T(z_1,z_2,...,z_{i-1},x,z_{i+1},...,z_n)=T(z_1,z_2,...,z_{i-1},y,z_{i+1},...,z_n).\]
    \item There is a weak near unanimity (or WNU) polymorphism $W\in \pol(\s D)$, i.e. an operation satisfying 
    \[W(y,x,x,...,x)=W(x,y,x,...,x)=W(x,x,y,...,x)=...=W(x,x,x,...,y)\]
    \item For all but finitely many primes $p$, $\s D$ has cyclic polymorphism of arity p, i.e. there is $C\in\pol(\s D)$ satisfying
    \[C(x_1,x_2,x_3,....,x_{p-1},x_p)=C(x_2,x_3,x_4,...,x_p,x_1)\]
    \item $\s D$ does not pp constructs every finite relational structure
\end{enumerate}
\end{thm}
\begin{proof}
Taylor proved an idempotent algebra $\bb A$ admits Taylor term if and only if $\var(\bb A)$ does not contain a projection algebra \cite{Taylor}. Since the Taylor identites are height 1, this implies $\pol(\s D)$ does not contain a Taylor term if and only if $\s D$ pp constructs every finite relational structure. The equivalence of Taylor terms and WNU terms is shown in \cite{WNU}, the equivalence with cyclic terms in \cite{Cyclic}, and the equivalence with Siggers terms in \cite{SiggersOG} and \cite{SiggersSharper}.
\end{proof}

Historically, the Taylor identities were the first to appear in the literature. In the 1970s, Motivated by questions in algebraic topology, Taylor showed that these identities characterize idempotent varieties which omit projection algebras. Bulatov, Jeavons, and Krokhin originally stated the algebraic CSP dichotomy conjecture in terms of Taylor identities, and many of the equivalent forms above were motivated by computational complexity questions. We include WNU operations here as they seem to show up most often in the literature.

The second class of structures we look at is described by a simple constraint propagation algorithm. Given an instance $\s X$, assign to each variable $x\in \s X$ a unary constraint $U_x$, initialized to $\s D$. This $U_x$ will represent to possible values $x$ can take. If we ever see that $R(x,y_1,...,y_n)$, but that $d\in U_x$ cannot be matched to elements of the $U_{y_i}$s to get an element of $R$, then remove $d$ from $U_x$. Repeat until this each $U_x$ stabilizes. If any $U_x$ is empty there is no solution. We say that a problem is width 1 if every instance with each $U_x$ nonempty has a solution. More formally:

\begin{dfn}\label{dfn: arc consistent}
Say that an instance $\s X$ of $\s D$ is \textbf{arc-consistent} if there are unary predicates $U_x\subseteq\s D$ for $x\in \s X$ which are pp definable in $\s D$ so that, if $(x_1,...,x_n)\in R^{\s X}$, then \[\pi_i(U_{x_1}\times...\times U_{x_n}\cap R^{\s D})=U_{x_i}.\]

A structure is \textbf{width 1} if every arc-consistent instance has a solution.
\end{dfn}

Width 1 structures will play an important role in Section \ref{section:classical}. An example of a width 1 structure is $\horn$. The arc consistency algorithm in this case amounts to classical unit propagation. A Theorem of Dalmau and Pearson algebraically characterizes width 1 structures.

\begin{thm}[\cite{Width1}] \label{thm:width1 characterization}
A structure $\s D$ is width 1 if and only if it admits a totally symmetric polymorphism of arbitrarily high arities.
\end{thm}

Indeed, $\horn$ has the $n$-ary ``and" as a polymorphism for all $n$. Arc-consistency only considers unary information, We can generalize this to consider binary interactions between constraints. For instance, the usual algorithm testing 2-colorability of graphs involves testing for odd length chains of binary relations.

\begin{dfn}\label{dfn:cycle consistent} A \textbf{path} in a relational structure $\s X$ is a sequence 
\[ P=(x_1, (e_1, R_1), x_2, (e_1, R_2), ..., x_n)\] if each $x_i\in X$, each $e_i\in R_i$, and for each $i$ there are $j_i,k_i$ so that $(x_i,x_{i+1})=\pi_{jk}e_i$. We say that $P$ has \textbf{coordinates} $(j_1,k_1,j_2,k_2,...,j_{n-1},k_{n-1})$. 

A \textbf{closed path} in a relational structure $\s X$ is a path $(x_1, (e_1, R_1),...,(e_{n-1}, R_{n-1}), x_n)$ with $x_1=x_n$.

An instance $\s X$ of $\s D$ is \textbf{cycle-consistent} if it is arc-consistent as witnessed by predicates $U_x$ and for any closed path $P$ with coordinates $(j_1,k_1,...,j_{n-1},k_{n-1})$,

\[(\forall a_1=a_{n}\in U_{x_1})(\exists a_2, a_3,...,a_{n-1})\; \bigwedge_{1\leq i\leq n-1} \left(U_{x_i}(a_i)\wedge \pi_{j_i,k_i}R_i(a_{i},a_{i+1})\right).\]

A structure has \textbf{bounded width} if every cycle-consistent instance has a solution.
\end{dfn}

Really, this definition hides a theorem. One could generalize arc-consistency in any number of ways to consider tuples of arbitrary large arity, but these all turn out to be redundant. See \cite{WidthCollapse} \cite{BddWidthSurvey}. Barto and Kozik characterized the bounded width CSPs, resloving another conjecture of Feder and Vardi.

\begin{thm}[\cite{BddWidthThm}] \label{thm:bddwidth characterization}
A structure is bounded width iff $\HSP(\pol(\s D))$ does not contain $\pol(\s F)$ for any finite field $\bb F$ iff $\HS(\pol(\s D))$ does not contain $\pol(\bb F)$ for any finite field $\bb F$.
\end{thm}

Cycle-consistency will be important in Section \ref{section:classical}, and the bounded width structures will play a role in Section \ref{section:effective}.

Our last example is a class of structures which is not closed under pp constructions, but still has an algebraic description.

\begin{dfn}\label{dfn: basis}
A \textbf{basis} for $\csp(\s D)$ is a set of structures $F$ so that $\s X\in \csp(\s D)$ if and only if no element of $F$ admits a homomorphism into $\s X$. Bases for $\csp_B(\s D)$ are defined mutatis matandus.
\end{dfn}

Bases are also called complete obstructing sets in the computer science literature. Descriptive set theorists are very interested in cases where $\csp_B(\s D)$ has a finite basis. The corresponding classical cases can be described in terms of polymorphisms with extra tolerance (which require reference to the relational structutre to define).

\begin{thm}[\cite{FiniteBasis}] \label{thm: finite basis}
For any structure $\s D$, $\csp(\s D)$ has a finite basis if and only if $\s D$ admits a polymorphism $f$ so that
\[f(y,x,x,...,x)=f(x,y,x,...,x)=f(x,x,y,...,x)=...=f(x,x,x,...,y)=x\]
and if $k-1$ of the tuples $e_1,...,e_k$ are in $R$ for some relation $R$ of $\s D$, then
\[f(e_1,...,e_k)\in R.\]
\end{thm}

We will not use this theorem in this paper, though bases will play a role in several places. Note that the above theorem does not refer only to $\pol(\s D)$. Indeed, it is not true that a finite basis for $\csp(\s D)$ will pass through a pp definition to give a finite basis for $\csp(\s E)$. 

\section{Simple constructions} \label{section:simple}

We would like a Borel analogue of the polynomial time reductions induced by pp constructions. Unfortunately, the finitary construction generally requires taking a quotient, which is not always possible in the descriptive setting (for instance, $\R/\Q$ is not a standard Borel space). Here we introduce a stronger notion of simulation, which we call a simple construction, which does not present this difficulty. And, we prove some combinatorial lemmas relating simple and pp constructions.

\begin{dfn}\label{dfn: simple dfns}
For $\s E$ and $\s D$ structures on the same domain, say that $\s D$ \textbf{simply defines} $\s E$ if every relation in $\s E$ can be written as an existentially quantified conjunction of relations in $\s D$. We call the formulas defining $\s E$ in terms of $\s D$ relations a \textbf{simple definition}.

Say that $\s D$ \textbf{simply interprets} $\s E$ if $\s E$ is a quotient of a structure simply definable in a power of $\s D$. We call the defining formulas a \textbf{simple interpretation}.

And, say that $\s D$ \textbf{simply constructs} $\s E$ if $\s E$ can be built from $\s D$ by a chain of simple interpretations, homomorphic equivalences, and singleton expansions of cores. We call this sequence a \textbf{simple construction}.
\end{dfn}

That is, simple constructions are pp constructions where we do not use equality (unless out structures come equipped with equality). Our first main lemma about simple constructions comes from a careful analysis of the proof that singleton expansions of cores do not change computational complexity.

\begin{lem}\label{lem:=c}
For any finite core structure $\s D$ and $c\in D$, the following predicate is simply definable in $\s D$:
\[(=_c):=\{(x,x): (\exists f\in \aut(\s D)) f(c)=x\}.\]
\end{lem}
\begin{proof}
Suppose $D=\{1,2,...,n\}$ and $c=1$. The predicate $H$ given by
\begin{align*}H(x_1,...,x_n)& :\lra \; d\mapsto x_d\mbox{ defines an endomorphism of } \s D \\
\; & \lra \bigwedge_{R}\bigwedge_{e\in R} R(x_{e_1},...,x_{e_n})\end{align*}
is simply definable. So, it suffices to show that
\[a =_c b\lra (\exists x_2,...,x_n)\; H(a, x_2,...,x_n)\wedge H(b, x_2,...,x_n). \] If $a=_c b$, then there is some isomorphism $f: \s D\rightarrow \s D$ with $f(c)=a$. Then the assignment $x_d=f(d)$ satisfies $H(a, x_2,...,x_n)\wedge H(b, x_2,...,x_n)$. 

Conversely, if $H(a, x_2,....,x_n)\wedge H(b, x_2, ...,x_n)$, then we have homomorphisms $f,g$ with $f(c)=a$, $g(c)=b$ and $f(d)=g(d)$ for $d\not=c$. These homomorphisms must be automorphisms since $\s D$ is a core. And if $a\not=b$, we must have that one of $f$ or $g$ is not onto, which is impossible. So, $a =_c b$.
\end{proof}
\begin{cor}\label{cor:transitive}
If $\s D$ has a transitive automorphism group and $\s D$ pp constructs $\s E$, then $\s D$ simply constructs $\s E$.
\end{cor}
\begin{proof}
We can replace $\s D$ by its core, which will still be transitive. Then, by the above lemma $\s D$ simply constructs $(=_c)= (=)$. 
\end{proof}

Our second main lemma comes from a careful analysis of the proof that $\pol(\s D)$ invariant relations are pp definable in $\s D$.

\begin{dfn}\label{dfn: implies equation}
Say that a relation $R$ implies an equation if $R(x_1,...,x_n)$ implies $x_i=x_j$ for some $i,j$, i.e. $\pi_{ij}(R)\subseteq (=).$
\end{dfn}

\begin{lem}\label{lem:equation free}
If $R$ does not imply any equations and $R$ is pp definable in $\s D$, then $R$ is simply definable in $\s D$.
\end{lem}

\begin{proof}
Let $M$ be a matrix whose rows are the tuples in $R$, and let $\sigma_i$ be the $i^{th}$ column of $M$. Note that $R$ not implying any equations means that the $\sigma_i$s are all distinct. Since $R^{\s D^n}(\sigma_1,...,\sigma_k)$ holds in the product, if there is a polymorphism $f\in\pol(\s D)$ so that $f(\sigma_i)=x_i$, then $R(x_1,...,x_k)$. Conversely, if $R(x_1,...,x_k)$, say $(x_1,...,x_k)$ is the $j^{th}$ row of $M$, then the projection $\pi_j$ is a polymorphism with $\pi_j(\sigma_i)=x_i$. Thus,
\[R(x_1,...,x_k)\lra (\exists f\in \pol(\s D))\;f(\sigma_i)=x_i.\]
Since the $\sigma_i$ are all distinct, we can eliminate equality in the above definition by replacing each instance of $\sigma_i$ with $x_i$. More formally, Suppose $\s D^k=\{\sigma_1,...,\sigma_k, \sigma_{k+1},...,\sigma_m\}$, and let \[P(x_1,...,x_m):\lra \sigma_i\mapsto x_i\mbox{ is a polymorphism of }\s D.\] Then $P$ is simply definable, and we have a simple definition of $R$ given by
\[R(x_1,...,x_k)\lra (\exists x_{k+1},...,x_{m})\; P(x_1,...,x_m).\]
\end{proof}

\begin{cor}\label{cor:HS thing}
If no relation in $\s E$ implies an equation and $\pol(\s E)\in \HS(\pol(\s D))$, then $\s D$ simply constructs $\s E$.
\end{cor}
\begin{proof}
If $\s E$ is in $\HS(\pol(\s D))$, then $\s E$ is a pp definable quotient of a pp definable substructure of $\s D$, i.e. there is some $U\subseteq \s D$ and some $f:U\rightarrow \s E$ with $U$ and $f\inv(R)$ pp definable for each relation $R$ in $\s E$. No unary predicate can imply an equation, and since no $R$ in $\s E$ implies an equation neither does any $f\inv(R)$. So by our lemma, all of these are in fact simply definable in $\s D$.
\end{proof}
\begin{cor}\label{cor:not bdd width}
If $\s D$ is not bounded width, then $\s D$ simply constructs $\bb F(3)$ for some finite field $\bb F$.
\end{cor}

We're ready to prove our descriptive analog of Bulatov and Jeavons's theorem about polynomial time reductions. Since we have not specified our coding, we will leave it to the reader to verify that the construction below is $\Delta^1_1$ in the codes, though this is straightforward using the coding described in the appendix and Lemma \ref{coding lemmas}. 

\begin{thm}\label{thm:simple reduction}
If $\s D$ simply constructs $\s E$, then $\csp_B(\s E)$ Borel reduces to $\csp_B(\s D)$. In fact, there are maps $F, G,$ and $H$ which are $\Delta^1_1$ in the codes so that:
\begin{enumerate}
    \item If $\s X$ is an instance of $\s E$ then $F(\s X)$ is an instance of $\s D$
    \item If $g$ is a solution to $\s X$, then $G(g)$ is a solution to $F(\s X)$
    \item If $h$ is a solution to $F(\s X)$, then $H(h)$ is a solution to $\s X$. 
\end{enumerate} And there is some finite $N$ so that, if each $x\in \s X$ appear in fewer than $\kappa$ tuples in relations $\s X$, then each $y\in F(\s X)$ appears in fewer than $N\times\kappa$ tuples in $F(\s X)$.
\end{thm}
This last clause means that $F$ sends bounded degree, locally finite, and locally countable instances to the same.
\begin{proof}
It suffices to consider two cases: $(1)$ $\s D$ is a core and $\s E$ is an expansion of $\s D$ by singleton unary predicates, and $(2)$ $\s E$ is simply interpretable in $\s D$.

For $(1)$ the main difficulty is controlling degrees, i.e. meeting the last clause of the theorem. We may assume $(=_c)$ is in the signature of $\s D$ for each $c\in D$ and $\s E=\s D\cup\{U_d\}$, where $U_d(x)\lra x=d$. If we did not want to worry about degree, we could let $F(\s X)$ be $\s X$ along with a copy of $\s D$ and extra relations saying $x=_d d$ whenever $U_d^{\s X}(x)$. Then any solution to $\s X$, say $g$, extends to a solution $G(g)$ of $F(\s X)$ by setting $G(g)(c)=c$ for $c\in \s D$. And if $h$ is a solution to $F(\s X)$, then $h$ restricts to an automorphism $f$ on the copy of $\s D$. So, we have a solution to $\s X$ given by $H(h)=f\inv\circ h$ on $\s X$.

In the construction we just sketched, the copy of $\s D$ ends up having quite large degree. To correct this, we give each variable in $\s X$ its own copy of $\s D$. So, define $F(\s X)$ as follows: $F(\s X)$ has domain $X\sqcup (X\times D)$ and the following relations
\begin{itemize}
    \item For $R$ a relation in $\s D$ with $R\not=(=_c)$ for any $c$
    \[R^{F(\s X)}=R^{\s X}\sqcup \{((x,a_1),...,(x,a_k)): x\in X, R^{\s D}(a_1,...,a_k)\}\]
    \item For $c\not=d$
    \[(=_c)^{F(\s X)}=(=_c)^{\s X}\sqcup\{((x,c), (y,c)): (x,y)\mbox{ share some relation in }\s X\}\]
    \item And,
    \[(=_d)^{F(\s X)}=(=_d)^{\s X} \sqcup\{(x,(x,d)): U_d^{\s X}(x)\}\]
\end{itemize}


Similar to the above, any solution $g$ to $\s X$ gives a solution $G(g)$ to $F(\s X)$ equal to the identity on any copy of $\s D$. And if $h$ is a solution to $F(\s X)$, $\tilde h_x(d):= h(x,d)$ defines an automorphism of $\s D$ for each $x$. Further $\tilde h_x=\tilde h_y$ whenever $x,y$ share some relation in $\s X$, so we get a solution to $\s X$ $H(h)(x)=(\tilde h_x\inv \circ h)(x)$.

For (2), the construction of $F(\s X)$ is identical to the Bulatov and Jeavons construction, but we use the Luzin--Novikov theorem to find $G$. Suppose we are given the following data from the definition of simple interpretation:
 
 \begin{enumerate}
 \item An onto function $c: A\rightarrow E$ with a simple definition for $A\subseteq D^n$: \[\bar x\in A\lra (\exists \bar z) \bigwedge_i \alpha_{A,i}(\bar x, \bar z)\]
 \item For each relation $R\in \s E$ a simple definition for $c\inv(R)$:
 \[\bar x_1,...,\bar x_k\in c\inv(R)\lra (\exists \bar z)\bigwedge_i \alpha_{R, i}(\bar x_1,...,\bar x_k, \bar z). \]
 \end{enumerate}
 we construct an instance $F(\s X)$ of $\s D$ as follows. For each variable $x\in X$ introduce tuples of variables $\bar y_x$ (with arity $n$ as in item 1 above) and $\bar z_{A, x}$ (with the same arity as $\bar z$ in the definition of item 1), and for each tuple $(x_1,...,x_n)\in R^{\s X}$ introduce tuples $\bar z_{R, x_1,...,x_n}$ (with the same arity as in the defintion in item 2). For each variable $x$ add in the relations
 \[\alpha_{A,i}(\bar y_x, \bar z_{A,x})\] and for each $(x_1,...,x_n)\in R^{\s X}$ add in relations
 \[\alpha_{R,i}(\bar y_{x_1},...,\bar y_{x_n}, \bar z_{R,x_1,...,x_n}).\] Given a solution $g$ to $F(\s X)$ we get a solution $G(g)$ to $\s X$ given by $G(g)(x)=c(g(\bar y_x))$ (where we apply $g$ coordinatewise to $\bar y_x$). And if we have a solution $g$ to $\s X$, we can get a solution to $F(\s X)$ by using the Lusin--Novikov theorem to choose values from $c\inv (g(x))$ for each $\bar y_z$ and to choose values for each $\bar z_{A, x}$ and $\bar z_{R, \bar x}$ from the witnesses to \[(\exists \bar z) \bigwedge_i \alpha_{A,i}(\bar y_x, \bar z)\quad\mbox{ and }\quad(\exists \bar z)\bigwedge_i \alpha_{R, i}(\bar y_{x_1},...,\bar y_{x_k}, \bar z). \]

\end{proof}

If $\s D$ has equality in its signature, then the distinction between simple and pp constructions collapses and we get the following:

\begin{cor}\label{cor:equality closure}
Suppose $\s D$ is a structure which includes equality as a relation and $\s E$ is pp constructible in $\s D$. If $\csp_B(\s D)$ is
 $\bp11$, effectivizable, or essentially classical then so too is $\csp_B(\s E)$.
\end{cor}

It follows that these classes all have algebraic characterizations when restricted to structures with equality, though the identities involved may be quite complex.

\section{Intractable CSPs} \label{section:intractable}

One of the consequences of the CSP dichotomy theorem is that (assuming P$\not=$NP) any structure with $\csp(\s D)$ complete for polynomial time reductions is in fact maximal for pp constructions. Using a lemma from the previous section and a lemma of Bulatov and Jeavons, we can show that, in fact, any such structure is maximal for simple constructions. This, along with Todorcevic and Vidnyanszky's $\bs12$-completeness theorem for 3-coloring \cite{TV}, gives a number of new projective complexity bounds.

\begin{lem}[Bulatov--Jeavons \cite{BJHS}] \label{lem: HS intractable}
If $\s D$ is intractable and includes all singleton unary predicates, then $\HS(\pol(\s D))$ contains a projection algebra.
\end{lem}
\begin{proof}
This lemma appears in a seemingly unpublished technical report, so we give a sketch of their proof here. 

Let $\bb A=\pol(\s D)$ and let $\bb B\leq \bb A^n$ be a subalgebra with some homomorphism $f:\bb B\rightarrow \bb X$ onto a nontrivial projection algebra. 
For $I\subseteq \{1,...,n\}$ and $\bar a\in A^{I}$, let $\bb B\res(I,a)=\pi_{ I}=\bar a$. Fix $I$ maximal with $f$ not constant on $\bb B\res(I,\bar a)$. Without loss of generality, we may assume $I=\{1,...,k\}$ and $\bar a=(a_1,...,a_k)$.

Since $\s D$ contains all unary predicates, $\bb A$ is idempotent. Thus, $\bb B'=\bb B\res(I, \bar a)$ is a subalgebra of $\bb B$. And we have that $f(\bb B')$ is a nontrivial projection algebra.

Let $\bb A'=\pi_{k+1}(\bb B')$. If there is some $a\in \bb A'$ with $f$ is not constant on $\pi_{k+1}\inv(a)\cap \bb B'$, then $I$ is not maximal. So, $f$ factors through $\pi_{k+1}$ and a nontrivial subalgebra of $\bb X$ is an image of $\bb A'$ as well.
\end{proof}
\begin{cor}\label{cor: intractable simply constructs all}
If $\s D$ is intractable, $\s D$ simply constructs all structures.
\end{cor}
\begin{proof}
We may assume $\s D$ is a core with all singleton unary relations. Then, by the lemma above, $\pol(3\sat)\in \HS(\pol(\s D))$. Since $3\sat$ does not imply equations, by lemma \ref{lem:equation free}, $\s D$ simply constructs $3\sat$. And since $3\sat$ simply defines equality and is maximal for pp constructions, $\s E$ simply constructs all structures.
\end{proof}
 Since there is a $\bs12$-complete Borel CSP, we get the following:
\begin{thm}\label{thm:intractable}
If $\s D$ is intractable, then $\csp_B(\s D)$ is $\bs12$-complete.
\end{thm}
\begin{proof}

Todocevic and Vidny\'ansky showed that $\csp_B(K_3)$ is $\bs12$-complete \cite{TV}. By the above theorem, if $\s D$ is instractable, then $\csp_B(K_3)$ Borel reduces to $\csp_B(\s D).$
\end{proof}
\begin{cor}[P$\not=$NP]\label{cor: p np intractable}
If $\csp(\s D)$ is NP-complete, then $\csp_B(\s D)$ is $\bs12$-complete. 
\end{cor}

In fact, since the instances from the Todorcevic--Vidnyanszky theorem are locally finite, these problems are $\bs12$-complete even when we restrict to locally finite instances. This can be improved to bounded degree by recent work of Brandt, Chang, Greb\'ik, Grunau, Rozho\v n, and Vidny\'anszky \cite{HomomorphismGraphs}. We can use this theorem to lift results from finitary complexity theory wholesale. For instance, the Hell--Ne\v set\v ril theorem \cite{HellNesetril} yields:

\begin{thm}\label{thm:hellnesetril}
For a simple graph $G$, the following are equivalent
\begin{enumerate}
    \item $G$ is bipartite
    \item $G$ is tractable
    \item $\csp_B(G)$ is effectivizable
    \item $\csp_B(G)$ is $\bp11$
    \item $\csp_B(G)$ is not $\bs12$-complete
\end{enumerate}
\end{thm}
\begin{proof}
$(1)\Rightarrow (2)$ is an unpublished result of Louveau and also follows from Theorem \ref{thm: dd effective}; $(3)\Rightarrow (4)\Rightarrow (5)$ are all clear; $(5)\Rightarrow (2)$ is the theorem above. The equivalence of $(1)$ and $(2)$ is essentially the classical Hell--Ne\v setril theorem. For completeness, we sketch a proof below.

If $G$ is bipartite, the core of $G$ is $K_2$, so $G$ is tractable. If $G$ is tractable, then it admits a cyclic polymorphism of all large enough prime arities. If $G$ is not bipartite, then there is some closed walk of large prime length $p$, say $(x_1,...,x_p)$. If $c$ is a cyclic polymorphism of arity $p$, then there is a loop at $c(x_1,...,x_p)=c(x_2,...,x_p, x_1)$, but $G$ is supposed to be a simple graph.
\end{proof}

A more general version for so-called smooth digraphs is given in Corollary \ref{cor: smooth digraphs}. We can also generalize Todocevic and Vidny\'anszky's theorem on graph coloring to hypergraphs

\begin{cor} \label{cor: hypergraphs}
For any arity $k$ and any number $n\geq 2$, the problem of Borel $n$-coloring $k$-ary hypergraphs is $\bs12$-complete.
\end{cor}
\begin{proof}
The hypergraph coloring problem is equivalent to the $\csp(\s D)$ for $\s D=(\{1,...,n\}, \{(x_1,...,x_k): \neg(x_1=...=x_k)\})$. These are known to be intractable. 
\end{proof}

And, we can find some exotic examples of $\bs12$-complete problems.

\begin{cor}\label{cor:special triad}
The directed graph shown in Figure \ref{figure:specialtriad} has a $\bs12$-complete Borel CSP.

\begin{figure}[h]
    \centering

\begin{tikzpicture}[rotate=90]
\tikzstyle{vtx}=[circle, draw=black, fill=black, inner sep=0pt, minimum size=4pt]
\node[vtx] (0) at (0,0) {};
\node[vtx] (1) at (-1,1) {};
\node[vtx] (2) at (-1,2) {};
\node[vtx] (3) at (-1,3) {};
\node[vtx] (4) at (-2,2) {};
\node[vtx] (5) at (-2,3) {};
\node[vtx] (6) at (-2,4) {};
\node[vtx] (7) at (-3,3) {};
\node[vtx] (8) at (-3,2) {};
\node[vtx] (9) at (-3,1) {};
\node[vtx] (10) at (-3,0) {};
\node[vtx] (11) at (1,1) {};
\node[vtx] (12) at (1,2) {};
\node[vtx] (13) at (2,1) {};
\node[vtx] (14) at (2,2) {};
\node[vtx] (15) at (2,3) {};
\node[vtx] (16) at (2,4) {};
\node[vtx] (17) at (3,3) {};
\node[vtx] (18) at (3,2) {};
\node[vtx] (19) at (3,1) {};
\node[vtx] (20) at (3,0) {};
\node[vtx] (21) at (-1,-1) {};
\node[vtx] (22) at (-1,-2) {};
\node[vtx] (23) at (-1,-3) {};
\node[vtx] (24) at (0,-2) {};
\node[vtx] (25) at (0,-1) {};
\node[vtx] (26) at (1,-2) {};
\node[vtx] (27) at (1,-3) {};
\node[vtx] (28) at (1,-4) {};
\node[vtx] (29) at (2,-3) {};
\node[vtx] (30) at (2,-2) {};
\node[vtx] (31) at (2,-1) {};
\node[vtx] (32) at (2,0) {};

\path[draw=black, thick, ->] (0) -- (1) ;
\draw[draw=black, thick, ->] (1) -- (2) ;
\draw[draw=black, thick,  ->] (2) -- (3);
\draw[draw=black, thick,  ->] (4) -- (3);
\draw[draw=black, thick,  ->] (4) -- (5);
\draw[draw=black, thick,  ->] (5)-- (6);
\draw[draw=black, thick,  ->] (7) -- (6);
\draw[draw=black, thick,  ->] (8) -- (7);
\draw[draw=black, thick,  ->] (9) -- (8);
\draw[draw=black, thick,  ->] (10)--(9);

\draw[draw=black, thick,  ->] (0) -- (11);
\draw[draw=black, thick,  ->] (11) -- (12);
\draw[draw=black, thick,  ->](13) -- (12);
\draw[draw=black, thick,  ->] (13) -- (14);
\draw[draw=black, thick,  ->] (14) -- (15);
\draw[draw=black, thick,  ->] (15) -- (16);
\draw[draw=black, thick,  ->] (17) -- (16);
\draw[draw=black, thick,  ->] (18) -- (17);
\draw[draw=black, thick,  ->](19) -- (18);
\draw[draw=black, thick,  ->] (20) --(19);

\draw[draw=black, thick,  ->] (0) --(21);
\draw[draw=black, thick,  ->](21) -- (22);
\draw[draw=black, thick,  ->](22) -- (23);
\draw[draw=black, thick,  ->](24)--(23);
\draw[draw=black, thick,  ->](25)--(24);
\draw[draw=black, thick,  ->](25)--(26);
\draw[draw=black, thick,  ->](26)--(27);
\draw[draw=black, thick,  ->](27)--(28);
\draw[draw=black, thick,  ->](29)--(28);
\draw[draw=black, thick,  ->](30)--(29);
\draw[draw=black, thick,  ->](31)--(30);
\draw[draw=black, thick,  ->](32)--(31);
\end{tikzpicture}
\caption{An intractable directed graph}\label{figure:specialtriad}
\end{figure}
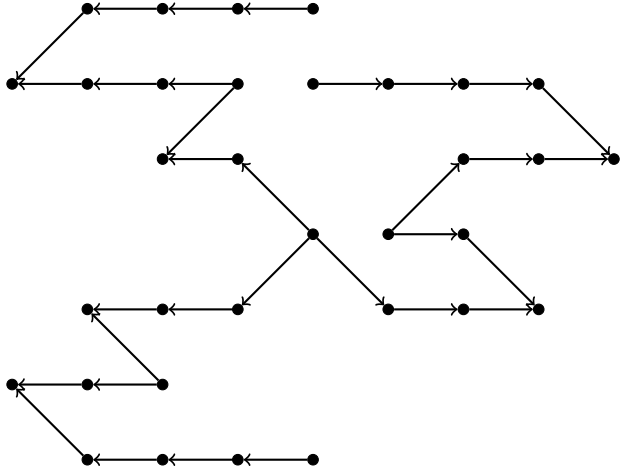
\end{cor}
\begin{proof}
Barto, Kozik, Mar\'oti, and Niven showed that this structure is intractable \cite{SpecialTriads}.
\end{proof}

We give a last example relating to the theory of local problems (also called locally checkable labelling problems or distributed constant time problems, see \cite{LocalSurvey}).

\begin{dfn}\label{dfn: graph with parallels}
A \textbf{graph with parallels} is a graph $G$ with a specified set of pairs of directed edges $P$. An \textbf{orientation} of a graph with parallels $(G,P)$ is an orientation $o$ of $G$ so that $e\in o$ iff $f\in o$ whenever $(e,f)\in P$. And orientation is \textbf{balanced} if the indegree of every vertex is the same as the outdegree.
\end{dfn}

\begin{cor}\label{cor:balanced orientations}
The set of Borel 4-regular graphs with parallels which admit a Borel balanced orientation is $\bs12$-complete.
\end{cor}
\begin{proof}

Given a vertex $v$ incident to (undirected) edges $e_1,e_2,e_3,e_4$ and and orientation $o$, let $f(e_i,v,o)$ be $1$ if $e_i$ is oriented toward $v$ and $-1$ otherwise. Then $o$ is balanced if and only if $\sum_i f(e_i,v,o)=0$ for all $v$, and this is true even if the sum is taken mod 3. This gives us a way to code 3 variable linear equations with nonzero solutions, a problem which turns out to be NP-complete.

Consider $\bb F_3^\times(3)$ the structure on $\{-1,1\}$ equipped with all relations of the form
\[a_1z_1+a_2z_2+a_3z_3+a_4=0\] for $a_1,a_2,a_3,a_4\in \{-1,1\}.$ It is straightforward to check that $\bb F_3^\times(3)$ is intractable (using, for instance, Schaefer's theorem). Write $R_{a_1,a_2,a_3,a_4}$ for the relation $\{(x,y,z): a_1x+a_2y+a_3z+a_4=0\}$. 

Given an instance $\s X$ of $\bb F_3^\times(3)$, define a graph with parallels $f(\s X)$ as follows. For a directed edge $(x,y)$, write $-1\cdot(x,y)$ for $(y,x)$. Let $T_4$ be the infinite 4-regular tree, and fix some vertex $v_0\in T_4$ incident to (directed) edges $e_i=(v_0,v_i)$ for $i=1,2,3,4$. Then the graph of $f(\s X)$ is a large disjoint union of copies of $T_4$:
\[V=\{(v,x): v\in T_4,x\in X\mbox{ or }x\in R^{\s X}\mbox{ for some }R\}\]
\[(v,x) E(v',x'):\lra v E v'\mbox{ and }x=x'.\]

And, the parallels of $f(\s X)$ are as follows:

\begin{enumerate} 
\item For $e=(x_1,x_2,x_3)\in R_{a_1,...,a_4}^{\s X}$, $((v_0, x_i), (v_1,x_i))$ is parallel to $a_i\cdot ((v_0,e), (v_i,e))$,

\item For $e\in R^{\s X}_{a_1,...,a_4}$ and $e'\in R^{\s X}_{b_1,...,b_4}$, $((v_0, e), (v_4,e))$ is parallel to $(a_4b_4)\cdot ((v_0, e'),(v_4,e'))$
\end{enumerate}

Suppose $o$ is a balanced orientation of $f(\s X)$. Because of the parallels in item (2), for $e\in R_{a_1,...,a_4}$, whether $((v_0, e),(v_4,e))$ is in $o$ or not only depends on $a_4$. Possibly replacing $o$ with its opposite orientation, we may assume $((v_0,e), (v_4,e))\in o$ if and only if $a_4=1$. For any edge $e$ in the graph of $f(\s X)$ define 
\[\tilde g(e)=\left\{\begin{array}{ll} 1 & e\in o \\ -1 & \mbox{otherwise} \end{array}\right.\] Then, I claim that $\s X$ has a solution given by $g(x)=\tilde g((v_0,x),(v_1,x))$. Indeed, if $e=(x_1,x_2,x_3)\in R^{\s X}_{a_1,a_2,a_3,a_4}$, then 
by the parallels in item (1) above $g(x_i)=a_i\cdot \tilde g((v_0, e),(v_i,e))$. By our normalization, $\tilde g((v_0, e),(v_4, e))=a_4$. And, since $o$ is balanced, $\sum_i g((v_0, e), (v_i,e))=0$.

Conversely, if $\s X$ has a solution $g$, we get a partial balanced orientation $o$ of $f(\s X)$ by setting 
\[((v_0,x),(v_1,x))\in o\lra g(x)=1\]
for $x\in \s X$, and for $e=(x_1,x_2,x_3)\in R_{a_1,...,a_4}^{\s X}$ and $i=1,2,3$,
\[((v_0,e),(v_i, e))\in o\lra a_ig(x_i)=1\]
and
\[((v_0,e),(v_2,e))\in o\lra a_4=1.\]

Then $o$ assigns an orientation to every edge which occurs in the set of parallels of $f(\s X)$, and every vertex is incident to either 1 or 4 edges assigned an orientation by $0$. Then since the graph of $f(\s X)$ is acyclic and has a smooth connectedness relation, $o$ extends to a balanced orientation on all of $f(\s X)$.
\end{proof}

Of course all of the complexity in the above construction is carried by the parallels. We would like to compute the complexity of graphs with balanced orientations, or of Borel local problems in general, but these are not in general exactly equivalent to CSPs. Indeed every NP-language is equivalent to a local problem.

\begin{prb}\label{prb: local}
Is the Borel version of every NP-complete local problem $\bs12$-complete?
\end{prb}


\section{Essentially classical CSPs} \label{section:classical}

From the point of view of descriptive set theory, essentially classical structures are trivial. Interestingly, though, these structures turn out to be exactly the width 1 structures from computer science, i.e. those solved by arc-consistency. We first give a characterization of arc-consistency which is convenient for reflection arguments.

\begin{dfn}\label{dfn:1minimal}
For an instance $\s X$ of $\s D$, a say $f:X\rightarrow \s P(D)$ is \textbf{closed} if, 
\[(\forall (x_1,...,x_n)\in R^{\s X}, i\leq n, a\in D)\; \bigwedge_{e\in \pi_i\inv(a)} e\not\in R\sm\Pi_j f(x_j) \Rightarrow a\in f(x_i).\] The \textbf{closure} of $f$ is function $\overline f$ obtained by iteratively adding points to each $f(x)$ to satisfy the above condition. 

A \textbf{good witness} for $\s X$ is a closed function $f$ so that, for all $x$, $f(x)\not=D$.
\end{dfn}

So, a structure is width 1 iff every structure with a good witness has a solution. And, $\s X$ has a good witness iff the closure of $\emptyset$ is a good witness. Recall Theorem \ref{thm:width1 characterization} says that a structure is width 1 if and only if it has a totally symmetric polymorphism of arbitrarily high arity.

\begin{thm}\label{thm:width1 implies classical}
If $\s D$ is a finite width 1 relational structure, then $\csp_B(\s D)$ is essentially classical and effectivizable.
\end{thm}
\begin{proof}

Suppose $\s X$ is an arc-consistent $\Delta^1_1$ instance of $\s X$. We first show that $\s X$ has a $\Delta^1_1$ 1-minimal witness, then we mimic the proof that $\s X$ has a solution to get a $\Delta^1_1$ solution.

Write $\Phi(f)$ to mean that $\overline f$ is a good witness for $\s X$. Note that, if $f$ is $\Sigma^1_1$, so is $\bar f$. Also $\Phi$ is $\Pi^1_1$ on $\Sigma^1_1$, so the first reflection theorem says any $\Sigma^1_1$ set satisfying $\Phi$ is contained in a $\Delta^1_1$ set satisfying $\Phi$. Since taking the closure and applying the reflection theorem can be done uniformly in the codes we have an increasing $\Delta^1_1$ sequence of sets
\[f_0=\emptyset\subseteq \overline f_0\subseteq f_1 \subseteq \overline f_1\subseteq  f_2...\] where each $A_i$ is $\Delta^1_1$ and satisfies $\Phi$. Then $\tilde f:=\bigcup_i f_i$ is a $\Delta^1_1$ good witness for $\s X$.

 Let $N$ be the largest arity of a relation in $\s D$ and let $T$ be a totally symmetric polymorphism of $\s D$ of arity $n\geq |\s D|\times N$. By the Lusin--Novikov theorem, there is some $f: \s X\rightarrow \s D^{n}$ so that \[D\sm \tilde f(x) =\{f(x)_1,...,f(x)_n\}.\] I claim that $T\circ f$ is a solution to $\s X$. To see this, suppose $(x_1,...,x_k)\in R^{\s X}$, let $M$ be an $k\times n$ matrix whose columns are in $R$ and whose rows enumerate $U_{x_i}$, and let $\sigma_i$ be the $i^{th}$ row of $M$. This possible since $n\geq |D|\times k$ and $\tilde f$ is a good witness. Then by total symmetry $T\circ f(x_i)=T(\sigma_i)$ and since $T$ is a polymorphism, ${(T\circ f(x_1),..., T\circ f(x_k))\in R.}$
\end{proof}

The converse requires a little more work. A theorem of Feder and Vardi says that a structure is width 1 if and only if it admits a basis of acyclic structures. So, if $\s D$ is not width 1, there is some unsolvable instance $\s X$ where every acyclic lift of $\s X$ has a solution. We sharpen this a bit to find an instance $\s X$ with a point $x$ so that any acyclic lift of $\s X$ has a solution, but there is an acyclic lift which has a solution that is constant on the fiber over $x$. Then a modification of the $G_0$ construction gives an acyclic lift of $\s X$ with no Borel solution.

\begin{dfn}\label{dfn: acyclic}
A \textbf{simple path} in $\s X$ is a path $(x_1, (e_1,R_1),...,(e_{n-1}, R_{n-1}), x_n)$ so that $(x_1, (e_1, R_1), x_2, (e_2, R_1),..., x_{n-1}, (e_n, R_n))$ is an injective sequence. 

A \textbf{cycle} is a simple path $(x_1, (e_1, R_1), x_2, (e_2, R_2),..., x_n)$ with $x_1=x_n$ A structure is $\textbf{acyclic}$ if it does not contain any cycles. 

A \textbf{lift} of $\s X$ is a structure $\s Y$ with a homomorphism $f: \s Y \rightarrow \s X$.
\end{dfn}

Note that if we impose any unary constraints on an acyclic structure it remains acyclic. The following lemma is essentially equivalent to tree duality for width 1 structures \cite{FederVardi}. We include a proof for completeness

\begin{lem}  \label{lem:tree duality}
An instance $\s X$ of $\s D$ is arc-consistent if and only if any acyclic lift of $\s X$ has a solution.
\end{lem}

\begin{proof}
First suppose that $\s X$ is arc-consistent as witnessed by predicates $U_x$, $U_{x_0}(a)$, $f:\s Y\rightarrow \s X$ is an acyclic lift of $\s X$. If $A\subseteq Y$ is connected, i.e. there is a simple path between any two points in $A$, and $R(y_1,...,y_n)$ holds with some $y_i\in A$ and $y_j\not\in A$, then in fact $\{y_1,...,y_n\}$ contains only one element of $A$ and no other relations of $\s Y$ meet $A\cup\{y_1,...,y_n\}$ at some element of $\{y_1,...,y_n\}$. Otherwise would be able to find a cycle in $\s Y.$ So, by arc-consistency, any partial solution to $\s Y$ with connected domain has an extension and a total solution must exist.

For the other direction, we build acyclic structures which encode the steps in the obvious algorithm for checking arc-consistency. Fix a linear order on pairs $(x,e, R)$ where $x$ is a coordinate of $e$ and $e\in R^{\s X}$. Say $(x,e,R)$ is bad for $\ip{U_x: x\in \s X}$ if $x$ is the $i^{th}$ coordinate of $e=(x_1,...,x_n)$ and $U_{x}\not=\pi_i(R\cap U_{x_1}\times...\times U_{x_n})$. For $i\in\N$ and $x\in \s X$ define $U^i_x$ inductively as follows:
\begin{enumerate}
    \item $U^0_X=D$ for all $x$
    \item If there is no $(x,e, R)$ bad for $\ip{U^i_x: x\in \s X}$ set $U^{i+1}_x=U^i_x$ for all $x$
    \item If $(x,e,R)$ is the least bad triple set $U^{i+1}_{x'}=U^i_x$ for $x\not=x'$ and set $U^{i+1}_x=\pi_i(R\cap U_{x_1}\times...\times U_{x_n}).$
\end{enumerate} By construction, $\s X$ is not arc-consistent if and only if $U^i_x=\emptyset$ for some $i$ and some $x$. We will show by induction that there are acyclic lifts $f^i_x: \s Y^{i}_{x}\rightarrow \s X$ and points $y^i_x\in (f^i_x)\inv(x)$ so that 
\[\{g(y^i_x): g\mbox{ is a solution to }\s Y^i_x\}=U^i_x.\]

For the base case, let $\s Y^0_x$ and have domain $\{y^i_x\}$, have $R^{\s Y^0_x}$ empty if $R$ has arity greater than 1, and $U(y^0_x)\lra U(x)$ for all unary relations, and set $f^i_x(y^i_x)=x.$ If $i$ is as in case $(2)$ of the induction or if $x$ is not in the least bad triple at step $i$, set $\s Y^{i+1}_x=\s Y^{i+1}_x$. Otherwise, suppose $(x, (x_1,...,x_n), R)$ is the least bad tuple at step $i$, $x=x_k$, and define
\[\s Y^{i+1}_x= \bigsqcup_{j=1}^n \s Y^i_{x_j}\]
with relations
\[R^{\s Y^{i+1}_x}=\{(y^i_{x_1},...,y^i_{x_n})\}\cup \bigsqcup_{j=1}^n \s R^{\s Y^i_{x_j}}\]
and $S^{\s Y^{i=1}_x}$ is just the union of the $S$ relations from each $\s Y^i_{x_j}$ for any relation $S$ besides $R$. Also, put $f^{i+1}_x(y)=f^i_{x_j}(y)$ if $y\in \s Y^i_{x_j}$ and put $y^{i+1}_x=y^i_x$. Then, $g$ is a solution $\s Y^{i+1}_x$ if and only if $g\res \s Y^i_{x_j}$ is a solution to $\s Y^i_{x_j}$ for each $j$ and $(g(y^i_{x_1}),...,g(y^i_{x_n}))\in R$. So by the inductive hypothesis
\begin{align*} 
\{g(y^{i+1}_x): g\mbox{ is a solution to }\s Y^{i+1}_x\} & = \pi_k (R\cap \Pi_j U^i_{x_j}) \\ & = U^{i+1}_{x}
\end{align*}

In particular, if $\s X$ is not arc-consistent, then $U^i_x=\emptyset$ for some $i,x$, so some acyclic lift of $\s X$ has no solution.
\end{proof}

To make sure the lift we construct in the next lemma is acyclic, we will need to introduce some new, simply definable relations. This is no issue since, whenever $\s D$ is width 1 or essentially classical, so is any structure which is simply definable in $\s D$. 

\begin{dfn}\label{dfn:simple closure}
For a structure $\s D$, let $\pf{\s D}$ be the structure with the same domain as $\s D$ equipped with every relation simply definable in $\s D$.
\end{dfn}

Note that $\pf{\s D}$ will not be finite even if $\s D$ is finite. But since it is simply definable in $\s D$, most theorems about finite structures will apply.

\begin{lem}\label{lem: sharper tree duality}
If there is an instance of $\s D$ which is arc-consistent but not cycle-consistent, then there is a finite arc-consistent instance $\s X$ of $\pf{\s D}$ with a finite acyclic lift $f:\s Y\rightarrow \s X$ and some $x\in\s X$ so that $\s Y$ has no solution which is constant on $f\inv(x)$.
\end{lem}
\begin{proof}

Similar to the arc-consistency algorithm sketched above, we can test cycle-consistency of an instance $\s X$ with arc-consistency witnesses $U_x$ by iteratively going through each closed path $P=(x_1, (e_1, R_1), ..., x_n, (e_n, R_n), x_1)$ with coordinates $(j_1, k_1,...,j_n, k_n)$, imposing a new unary constraint $U_P(x_1)$ on $x_1$ with
\[U_P(a_1):\lra (\exists a_2, a_3,...,a_n)\; \bigwedge_k \left( U_{x_k}(a_k) \wedge \pi_{i_k,j_k} R_k(a_k,a_{k+1})\right),\] and then refining the witnesses to arc-consistency. Stopping this process one step early we can assume $\s X$ is arc-consistent but $\s X\cup \{U_P\}$ is not arc-consistent.

By the previous lemma, there is an acyclic lift $f:\s Y\rightarrow\s X\cup{U_P}$ with no solution. We can convert $\s Y$ into an acyclic lift $f': \s Y'\rightarrow \s X$ with no solution which is constant on $f\inv(x_1)$ as follows. Note that, since $U_P^{\s X}=\{x_1\}$, $U_P^{\s Y}\subseteq f\inv(x)$. For each $y\in U_P^{\s Y}$ introduce new variables, $z_{2,y}$, $z_{3,y},...,z_{n,y}$, and impose constraints $\pi_{i_k,j_k}R_k(z_{k,y}, z_{k+1,y})$, $U_{x_1}(y)$, and $U_{x_k}(z_{k})$. And, extend $f$ to $f'$ by setting $f(z_{k,y})=x_k$ for each new variable $z_{k,y}$. Then $\s Y'$ is an acyclic lift of $\s X$. If $g$ is a solution to $\s Y'$ which is constant on $f\inv(x)$, then whenever $y\in U_{P}^{\s Y}$, $g(z_{2,y}),...,g(z_{n, y})$ witness that $g(y)\in U_P$, so $g$ restricts to a solution to $\s Y$, which is a contradiction.
\end{proof}

Now with $\s X$ and $\s Y$ as above, a simple modification of the $G_0$ construction gives an acyclic lift $\s Y_0$ of $\s X$ which contains many copies of $\s Y$ so that any Borel (in fact any Baire measurable) map $g: \s Y\rightarrow \s D$ must be constant on the fiber of $y$ in some copy of $\s Y$.

\begin{thm}\label{thm: classical implies width1}
If $\s D$ is finite and not width 1, then there is a Borel instance of $\s D$ with a solution but no Baire measurable solution.
\end{thm}
\begin{proof}
We may assume $\s D=\pf{\s D}$. Either $\s D$ is bounded width or, by Theorem \ref{thm:bddwidth characterization}, $\s D$ simply constructs $\bb F(3)$ for some finite field $\bb F$. In either case, there is an arc-consistent but not cycle-consistent instance of $\s D$. Then by the previous lemma, we can find $\s X$ arc-consistent, $f:\s Y\rightarrow \s X$ a finite acyclic lift, and $x\in \s X$ so that $\s Y$ has no solution which is constant on $f\inv(x)$. 

Let $f\inv(x)=\{x_1,...,x_n\}$ and define an $n$-ary relation $R$ by 
\[R(a_1,...,a_n):\lra x_i\mapsto a_i\mbox{ extends to a solution to }\s Y.\] Then, $R$ is simply definable in $\s D$, any acyclic instance of $R$ has a solution, and there is no solution to any instance of $R$ which constant on any tuple of $R$.

Fix a sequence $\ip{\sigma_i: i\in \N}$ with $\sigma_i\in \{1,...,n\}^i$ so that every string in $\{1,...,n\}$ extends to some $\sigma_i$. Let $\s Y_0$ be the instance of $R$ with domain $\{1,...,n\}^\omega$ and 
\[R^{\s Y_0}(s_1,...,s_n):\lra (\exists 1\leq i\leq n, t\in \{1,...,n\}^\omega) \bigwedge_{j=1}^{n} s_j=\sigma_i\fr j\fr t.\]

Since $\s Y_0$ is acyclic it has a solution. But suppose $g: \s Y_0\rightarrow \s D$ is Baire measurable. There is some $a\in \s D$ with $g\inv(a)$ nonmeager. Then $g\inv(a)$ is comeager in some basic neighborhood $N_{\sigma_i}$. Since the maps which cycle the $(i+1)^{th}$ coordinate of a sequence are self-homeomorphisms of $N_{\sigma_i}$, there is some $t$ so that $(\sigma_i\fr j\fr t)\in g\inv(a)$ for $j=1,...,n$. But then $g$ is constant on some tuple which is a contradiction.
\end{proof}
\begin{cor}
The set of essentially classical structures is decidable.
\end{cor}
\begin{cor}
If $\s D$ is essentially classical it is effectivizable.
\end{cor}
\begin{proof}
If $\s D$ is essentially classical, then by the above theorem $\s D$ admits a totally symmetric polymorphism. So by Theorem \ref{thm:width1 implies classical}, $\s D$ is effectivizable.
\end{proof}

\section{Effectivizable CSPs} \label{section:effective}

In this section, we give a number of examples of effectivizable structures. In particular we show that any structure with a dual discriminator polymorphism is effectivizable. This is modest progress, but it is enough to compute the complexity of any so-called smooth directed graph and any Boolean structure except $\bb F_2(3)$.

\begin{prop}[Folklore]
Suppose that $\s E$ has domain $D$ and a dual discriminator polymorphism. Then $\s E$ is simply definable in the structure $\s D$ with the following relations:

\begin{itemize}
    \item Every unary predicate
    \item For $a,b\in A$, each predicate\[R_{a,b}(x,y):\lra x=a\vee y=b\]
    \item For $f\in \operatorname{Sym}(A)$, each predicate
    \[R_f(x,y):\lra y=f(x)\]
\end{itemize}
\end{prop}
\begin{proof}
Let $T:D^3\rightarrow D$ be the dual discriminator. It is straightforward to check these relations are all preserved by $T$. Suppose that $R\subseteq D^n$ is preserved by $T$. We first show that $R$ is a conjunction of binary predicates. Suppose that $\pi_{i,j}(a)\in \pi_{i,j}(R)$. We check by induction on $|J|$ that $\pi_J(a)\in \pi_J(R)$ for any $J\subseteq \{1,...,n\}$. Pick some $J$ with $|J|=\ell+1$ and pick $J_1, J_2, J_3\subseteq J$ distinct with $|J_1|=|J_2|=|J_3|=\ell$. By induction, there are $b_1$, $b_2$, and $b_3\in R$ with $\pi_{J_i}(b_i)=\pi_{J_i}(a)$. Then, $ b:= T(b_1,b_2,b_3)\in R$ and for any $j\in J$, there are at least 2 values of $i$ with $j\in J_i$, meaning the majority of $b_1,b_2,b_3$ agree with $a$ in coordinate $j$. Thus $\pi_J(a)=\pi_J(b).$

Now suppose that $R$ is a binary relation which is invariant under $T$. We show that it is a conjunction of relations of the above form. Let $A=\pi_2(R)$ and $B=\pi_2(R)$. Note that if $a\in D$ can be paired with two different elements $b,b'$ so that $(a,b),(a,b')\in R$, then for any $(c,d)\in R$, we have
\[d((c,d),(a,b),(a,b'))=(a,d)\in R.\] So if $a$ can be paired with two different elements of $B$, then $a$ can be paired with with anything in $B$. Thus $R$ is the either $A\times B$ or the intersection of $A\times B$ with a relation of the form $R_\pi$ or $R_{a,b}$.
\end{proof}

Effectivization for $\csp_B(\s D)$ follows fairly easily from the main theorem of \cite{effectivization}.

\begin{thm}\label{thm: dd effective}
If $\s D$ has a dual discriminator polymorphisms, then $\csp_B(\s D)$ is effectivizable.
\end{thm}
\begin{proof}
We may assume that $\s D$ is of the form indicated in the previous proposition with domain $D=\{1,...,n\}$. Fix a $\Delta^1_1$ instance $\s X$ of $\s D$. For $f\subseteq \baire\times D$, write $\Phi(f)$ to mean
\begin{enumerate}
    \item $f$ is a partial function, i.e. $(\forall x,y,z)\; \neg\left((x,y),(x,z)\in f, y\not=z\right)$
    \item For all $x\in \s X$, unary predicates $U$ with $U^{\s X}(x)$, and $d\not\in U$, $(x,d)\not\in f$
    \item For all $x,y\in\s X$ and $a,b\in D$ with $(x,y)\in R^{\s X}_{a,b}$, if there is $c\not=a$ with $(x,c)\in f$ then $(y,b)\in f$ (and likewise if $(y,c)\in f$ for some $c\not=b$ then $(x,a)\in f$
    \item For all $x,y\in X$ and $g\in \operatorname{Sym}(D)$ with $(x,y)\in R^{\s X}_f$, if $(x,a)\in f$, then $(y,g(a))\in f$ (and likewise, if $(y,a)\in f$, then $(x,g\inv(a))\in f$).
\end{enumerate}
I claim that $\s X$ has a ($\Delta^1_1$) solution if and only if there is a sequence of ($\Delta^1_1$) sets $\ip{f_i:i\in \N}$ such that $\Phi(f_i)$ for all $i$ and $\baire=\bigcup_i \dom(f_i)$. Then since the properties above are all closure and independence properties the theorem follows by \cite[Theorem 3.6]{effectivization}.

If $f$ is a solution to $\s X$ then $f$ satisfies all of the above properties and $\dom(f)=\baire.$ Conversely, suppose $\ip{f_i:i\in\baire}$ is such a sequence and define 
\[n(x):=\min\{i: x\in\dom(f_i)\},\; \quad f(x):=f_{n(x)}(x).\] We check that $f$ preserves all of the relations in $\s D$.
\begin{itemize}
    \item If $U^{\s X}(x)$ for some unary $U$, then by property $(2)$ $U(f_i(x))$ holds for all $i$ with $x\in \dom(f_i)$. Thus $U(f(x))$.
    \item If $R_{a,b}^{\s X}(x,y)$, then suppose without loss of generality $n(x)\leq n(y)$. If $f_{n(x)}(x)=a$, then $f(x)=a$. If $f_{n(x)}(x)\not=a$, then by property (3) $(y,b)\in f_{n(x)}$, so $n(x)=n(y)$ and $f(y)=b$. In either case $R_{a,b}(f(x),f(y)).$
    \item If $R_g^{\s X}(x,y)$, then suppose without loss of generality $n(x)\leq n(y)$. If $(x,a)\in f_i$, then by property (4), $(y,g(a))\in f_i$, so $n(y)=n(x)$ and $R_g(f(x),f(y)).$
\end{itemize}
\end{proof}

We can now generalize Corollary \ref{thm:hellnesetril} to so-called smooth digraphs. If we could generalize this to all directed graph, then we would have Conjecture \ref{cnj: dichotomy} \cite{digraphreduction}.

\begin{cor}\label{cor: smooth digraphs}
If $\s D$ is a directed graph with no sources or sinks (these are sometimes called smooth digraphs), then $\csp_B(\s D)$ is effectivizable if and only if it is $\bp11$ if and only if $\s D$ is tractable.
\end{cor}
\begin{proof}
By a theorem of Barto, Kozik, and Niven \cite{SmoothDigraphs}, for such graphs either $\s D$ is intractable or the core of $\s D$ is a disjoint union of directed cycles. A disjoint union of cycles is the graph of a permutation, so any tractable smooth digraph admits a dual discriminator polymorphism.
\end{proof}

We can also compute the complexity of most Boolean sturctures.

\begin{cor}[c.f. Schaefer's theorem \cite{Schaefer}] \label{cor: schaefer}
If $\s D$ is a structure on $\{0,1\}$, one of the following holds:
\begin{enumerate}
    \item $\s D$ has a totally symmetric polymorphism and $\csp(\s D)$ is essentially classical
    \item $\s D$ is pp constructible in $2\sat$ and $\csp_B(\s D)$ is effectivizable
    \item $\s D$ is intractable and $\csp_B(\s D)$ is $\bs12$-complete, or
    \item $\s D$ is pp constructible in $\bb F_2$ and vice versa.
\end{enumerate}
\end{cor}
\begin{proof}
We'll take the opportunity to use a bit of a sledgehammer here. In the 1940's Post classified all clones on $\{0,1\}$, see \cite{Post}. By inspecting the minimal elements in the lattice of clones, one can see that $\s D$ falls into one of the following cases: 
\begin{enumerate}
    \item $\pol(\s D)$ contains a constant function, $\vee$, or $\wedge$
    \item $\pol(\s D)$ contains the majority function
    \item $\pol(\s D)\subseteq\ip{\neg}$ (the algebra generated by the negation function)
    \item $\pol(\s D)$ is one of the following: $\ip{x\oplus y\oplus z}$ or $\ip{x\oplus y \oplus z, \neg}$.
\end{enumerate} All of the operations in the first case are totally symmetric. In the second case, $\csp(\s D)$ is in fact pp definable in $2\sat$. Note that $\ip{\neg}=\pol(N)$, where $N$ is the not-all-equal predicate. So in the third case, $\s D$ is intractable. And either of the algebras in the last class correspond to structures which are equivalent $\bb F_2$.
\end{proof}

And, we have one last example.

\begin{thm} \label{thm: rps effective}
Let $\s D$ be the structure on $\{r,p,s\}$ equipped with all relations which are preserved by the rock-paper-scissors operation, $(\star)$ (see item (6) after Definition \ref{dfn: polymorphisms}). Then Borel solutions to locally countable instances of $\s D$ are effectivizable.
\end{thm}
\begin{proof}
$\s D$ is generated under pp definitions by the following relations: \begin{itemize}
    \item $R_\pi$, the graph of the cyclic permutation $\pi=(rps).$
    \item $R_{\star}(x,y,z):\lra x\in \{p,s\}\wedge (x=s\vee y=z)$.
\end{itemize}

Let $\s X$ be a locally countable $\Delta^1_1$ instance of $\s D$ with a Borel solution $\tilde f$.  Let $\s A$ be the 1-minimal closure of $\emptyset$. Since $\s X$ is locally countable, $A$ is $\Delta^1_1$. Let $U_x=\{d: (x,d)\not\in A\}$. Note that $U_x$ is a witness to arc-consistency for $\s X$. 

Say that $x$ is fixed if $|U_x|=1$ and critical if $|U_x|=2$ and free otherwise. Let $G$ be the weighted directed Borel graph with edges $(x,y)$ and $(y,x)$ of weight 0 whenever $R(z,x,y)$ with $z$ fixed and $p\in U_z$, edges $(x,y)$ of weight 1 whenever $R_\pi(x,y)$, and edges $(y,x)$ of weight -1 whenever $(x,y)$ is an edge of weight 1. Then, let $d(x,y)$ be the sum of weights along a directed path from $x$ to $y$ modulo 3 if such a path exists and $\infty$ otherwise. This is well defined since $\s X$ must be cycle-consistent. Note that if $d(x,y)<\infty$ and $x$ is free (or critical or fixed) then so is $y$. If $d(x,y)=i<\infty$ and $f$ is a solution to $\s X$, we must have $f(y)=\pi^i(x).$

Define $f$ as follows: 
\begin{itemize}
    \item for $x$ fixed, set $f(x)=a$ if $a\in U_x$
    \item for $x$ critical with $U_x=\{a,b\}$, set $f(x)=a\star b$
\end{itemize}

If $d(x,y)=i$ and $f(x)=a$ (in particular $x$ is not free), then $U_y=\pi^i(U_x)$, so $f(y)=\pi^i(x)$. And, if $R_{\star}(x,y,z)$, we must have that $x$ is not free. If $x$ is critical, then $U_x=\{p,s\}$, so $f(x)=s$ and this instance of $R_{\star}$ is satisfied by any extension of $f$ (in particular $f$ is a partial homomorphism). If $x$ is fixed, then this instance of $R_{\star}$ will be satisfied by an extension $g$ of $f$ if and only if $g(y)=g(z)$.

So, we want to find a $\Delta^1_1$ function $g$ on the free variables so that $d(x,y)=i$ implies $g(y)=\pi^i(g(x))$. We know that we have a Borel such function (namely the restriction of $\widetilde f$ to the free variables). And, such functions are effectivizable by the previous theorem.

\end{proof}

This last example is archetypal of bounded width structures, which can be solved by an intricate greedy algorithm. However Problem \ref{prb: bddwidth} remains open even if we restrict to locally countable instances. We end with a conditional result:

\begin{prop}
If every bounded width structure is effectivizable and every structure of the form $\bb F(3)$ is $\bs12$-complete, then the following are equivalent:
\begin{enumerate}
    \item $\csp_B(\s D)$ is effectivizable
    \item $\csp_B(\s D)$ is $\bp11$
    \item $\csp_B(\s D)$ is not $\bs12$-complete
\end{enumerate}
\end{prop}
\begin{proof}
$(1)\Rightarrow (2)\Rightarrow (3)$ is clear. If $(3)$ holds, then $\csp_B(\s D)$ must not pp construct $\bb F(3)$ for any finite field $\bb F$. In particular, $\pol(\bb F(3))\not\in HS(\pol(\s D))$, but then $\s D$ is bounded width, and $(1)$ holds.
\end{proof}

It is not clear how plausible the assumptions in this proposition are, but the resulting equivalences are appealing.

\bibliographystyle{hplain}
\bibliography{refs}

\begin{thebibliography}{10}

\bibitem{WidthCollapse}
Libor Barto.
\newblock The collapse of the bounded width hierarchy.
\newblock {\em Journal of Logic and Computation}, 26, 11 2014.

\bibitem{BddWidthThm}
Libor Barto and Marcin Kozik.
\newblock Constraint satisfaction problems of bounded width.
\newblock In {\em 2009 50th Annual IEEE Symposium on Foundations of Computer
  Science}, pages 595--603, 2009.

\bibitem{Cyclic}
Libor Barto and Marcin Kozik.
\newblock Absorbing subalgebras, cyclic terms, and the constraint satisfaction
  problem.
\newblock {\em Log. Methods Comput. Sci.}, 8, 2012.

\bibitem{SpecialTriads}
Libor Barto, Marcin Kozik, Mikl\'os Mar\'oti, and Todd Niven.
\newblock {CSP} dichotomy for special triads.
\newblock {\em Proceedings of The American Mathematical Society}, 137, 09 2009.

\bibitem{SmoothDigraphs}
Libor Barto, Marcin Kozik, and Todd Niven.
\newblock The {CSP} dichotomy holds for digraphs with no sources and no sinks
  (a positive answer to a conjecture of {Bang-Jensen and Hell}).
\newblock {\em SIAM J. Comput.}, 38:1782--1802, 01 2009.

\bibitem{PolymorphismSurvey}
Libor Barto, Andrei Krokhin, and Ross Willard.
\newblock {Polymorphisms, and How to Use Them}.
\newblock In Andrei Krokhin and Stanislav Zivny, editors, {\em The Constraint
  Satisfaction Problem: Complexity and Approximability}, volume~7 of {\em
  Dagstuhl Follow-Ups}, pages 1--44. Schloss Dagstuhl--Leibniz-Zentrum fuer
  Informatik, Dagstuhl, Germany, 2017.

\bibitem{reflections}
Libor Barto, Jakub Opršal, and Michael Pinsker.
\newblock The wonderland of reflections.
\newblock {\em Israel Journal of Mathematics}, 223, 10 2015.

\bibitem{Birkhoff}
Garret Birkhoff.
\newblock {The structure of abstract algebras.}
\newblock {\em Proceedings of the Cambridge Philosophical Society},
  31:433--454, 1935.

\bibitem{GaloisOG2}
V.~G. Bodnarchuk, L.~A. Kaluzhnin, V.~N. Kotov, and Boris~A. Romov.
\newblock Galois theory for {Post} algebras. ii.
\newblock {\em Cybernetics}, 5:531--539, 1969.

\bibitem{BddWidthSurvey}
Zarathustra Brady.
\newblock Examples, counterexamples, and structure in bounded width algebras.
\newblock {\em ArXiv}, abs/1909.05901, 2019.

\bibitem{HomomorphismGraphs}
Sebastian Brandt, Yi-Jun Chang, Jan Greb{\'i}k, Christoph Grunau, V{\'a}clav
  Rozhon, and Zolt'an Vidny'anszky.
\newblock On homomorphism graphs, 2021, arXiv: 2111.03683.

\bibitem{BulatovJeavonsKrokhin}
Andrei Bulatov, Peter Jeavons, and Andrei Krokhin.
\newblock Classifying the complexity of constraints using finite algebras.
\newblock {\em SIAM J. Comput.}, 34:720--742, 04 2005.

\bibitem{BulatovCSPDichotomy}
Andrei~A. Bulatov.
\newblock A dichotomy theorem for nonuniform {CSPs}.
\newblock {\em 2017 IEEE 58th Annual Symposium on Foundations of Computer
  Science (FOCS)}, pages 319--330, 2017.

\bibitem{BJHS}
Andrei~A. Bulatov and Peter Jeavons.
\newblock Algebraic structures in combinatorial problems, 2001.

\bibitem{digraphreduction}
Jakub Bulin, Dejan Delic, Marcel Jackson, and Todd Niven.
\newblock A finer reduction of constraint problems to digraphs.
\newblock {\em Log. Methods Comput. Sci.}, 11(4), 2015.

\bibitem{Width1}
Víctor Dalmau and Justin Pearson.
\newblock Closure functions and width 1 problems.
\newblock In {\em In CP 1999}, pages 159--173. Springer-Verlag, 1999.

\bibitem{FederVardi}
Tomás Feder and Moshe Vardi.
\newblock The computational structure of monotone monadic {SNP} and constraint
  satisfaction: A study through {Datalog} and group theory.
\newblock {\em SIAM J. Comput.}, 28:57--104, 01 1998.

\bibitem{GaloisOG1}
David Geiger.
\newblock {Closed systems of functions and predicates.}
\newblock {\em Pacific Journal of Mathematics}, 27(1):95 -- 100, 1968.

\bibitem{HellNesetril}
Pavol Hell and Jaroslav NeÅ¡etÅ™il.
\newblock On the complexity of {H}-coloring.
\newblock {\em Journal of Combinatorial Theory, Series B}, 48(1):92--110, 1990.

\bibitem{Jeavons}
Peter Jeavons.
\newblock On the algebraic structure of combinatorial problems.
\newblock {\em Theor. Comput. Sci.}, 200:185--204, 1998.

\bibitem{SiggersSharper}
Keith Kearnes, Petar Marković, and Ralph Mckenzie.
\newblock Optimal strong {M}al’cev conditions for omitting type 1 in locally
  finite varieties.
\newblock {\em Algebra universalis}, 72, 08 2014.

\bibitem{KST}
A.S Kechris, S~Solecki, and S~Todorcevic.
\newblock {Borel} chromatic numbers.
\newblock {\em Advances in Mathematics}, 141(1):1--44, 1999.

\bibitem{Ladner}
Richard~E. Ladner.
\newblock On the structure of polynomial time reducibility.
\newblock {\em J. ACM}, 22(1):155–171, jan 1975.

\bibitem{FiniteBasis}
Beno{\^i}t Larose, Cynthia Loten, and Claude Tardif.
\newblock A characterisation of first-order constraint satisfaction problems.
\newblock {\em 21st Annual IEEE Symposium on Logic in Computer Science
  (LICS'06)}, pages 201--210, 2006.

\bibitem{WNU}
Miklós Maróti and Ralph Mckenzie.
\newblock Existence theorems for weakly symmetric operations.
\newblock {\em Algebra Universalis}, 59:463--489, 12 2008.

\bibitem{moschovakis}
Y.N. Moschovakis.
\newblock {\em Descriptive Set Theory}.
\newblock Mathematical surveys and monographs. American Mathematical Society,
  2009.

\bibitem{Post}
Emil~Leon Post.
\newblock {\em The Two-Valued Iterative Systems of Mathematical Logic}.
\newblock London: Oxford University Press, 1941.

\bibitem{Schaefer}
Thomas~J. Schaefer.
\newblock The complexity of satisfiability problems.
\newblock In {\em Proceedings of the Tenth Annual ACM Symposium on Theory of
  Computing}, STOC '78, page 216–226, New York, NY, USA, 1978. Association
  for Computing Machinery.

\bibitem{SiggersOG}
Mark Siggers.
\newblock A strong {M}al’cev condition for locally finite varieties omitting
  the unary type.
\newblock {\em Algebra Universalis}, 64:15--20, 10 2010.

\bibitem{LocalSurvey}
Jukka Suomela.
\newblock Survey of local algorithms.
\newblock {\em ACM Comput. Surv.}, 45(2), mar 2013.

\bibitem{Taylor}
Walter Taylor.
\newblock Varieties obeying homotopy laws.
\newblock {\em Canadian Journal of Mathematics}, 29(3):498–527, 1977.

\bibitem{effectivization}
Riley Thornton.
\newblock {$\Delta^1_1$} effectivization in {Borel} combinatorics, 2021, arXiv:
  2105.04063.

\bibitem{TV}
Stevo Todorčević and Zoltán Vidnyánszky.
\newblock A complexity problem for {B}orel graphs.
\newblock {\em Inventiones mathematicae}, 226, 10 2021.

\bibitem{ZhukCSPDichotomy}
Dmitriy Zhuk.
\newblock A proof of {CSP} dichotomy conjecture.
\newblock {\em 2017 IEEE 58th Annual Symposium on Foundations of Computer
  Science (FOCS)}, pages 331--342, 2017.

\end{thebibliography}

\pagebreak
\appendix
\section{Edge coloring} \label{appendix:codings}

In this appendix, we prove that the set of Borel edge 3-colorable graphs is $\bs12$-complete. Since edge coloring involves a restricted class of instances, this does not follow from Theorem \ref{thm:intractable}, but the classical proof that edge 3-coloring is NP-complete still adapts to the Borel setting. We also take this as an opportunity to show in detail how to verify that a construction is $\Delta^1_1$ in the codes.

First, we recall the classical NP-completeness proof. Roughly, we will reduce 3SAT to 3 edge coloring by coding variables as pairs of edges, and the values that the variable can take will be coded into whether the corresponding edges can receive the same color. We need some lemmas.

\begin{lem}[Inverter lemma] \label{lem: inverter lemma}
There is a graph $I$ with distinguished edges $a,b,c,d,e$ so that a 3-coloring $f$ of $a,b,c,d,e$ extends to a coloring of $H$ if and only if one of the following holds:
\[f(a)=f(b) \quad\mbox{and }f(e)\not=f(c)\not=f(d)\not=f(e)\]

or

\[f(c)=f(d) \quad\mbox{and }f(e)\not=f(a)\not=f(b)\not=f(e)\]
\end{lem}
\begin{proof}
Such a graph is pictured below in Figure \ref{inverter}.

\end{proof}
\begin{figure}[h]
    \centering
    \begin{tikzpicture}
\tikzstyle{vtx}=[circle, draw=black, fill=black, inner sep=0pt, minimum size=4pt]
\node[vtx] (1) at (0,2) {};
\node[vtx] (2) at (0,0) {};
\node[vtx] (3) at (3,4) {};
\node[vtx] (4) at (6,2) {};
\node[vtx] (5) at (6,0) {};
\node[vtx] (6) at (1,2) {};
\node[vtx] (7) at (1,0) {};
\node[vtx] (8) at (2,1) {};
\node[vtx] (9) at (3,3) {};
\node[vtx] (10) at (4,1) {};
\node[vtx] (11) at (5,2) {};
\node[vtx] (12) at (5,0) {};

\path[draw=black, -] (1) -- (6) node[midway, above, color=black] {$a$} ;
\path[draw=black, -] (2) -- (7) node[midway, above, color=black] {$b$} ;
\path[draw=black, -] (3) -- (9) node[midway, left, color=black] {$e$};
\path[draw=black, -] (4) -- (11) node[midway, above, color=black] {$c$};
\path[draw=black, -] (5) -- (12) node[midway, above, color=black] {$d$};
\path[draw=black, -] (6) -- (8);
\path[draw=black, -] (6) -- (11);
\path[draw=black, -] (7) -- (10);
\path[draw=black, -] (7) -- (12);
\path[draw=black, -] (8) -- (9);
\path[draw=black, -] (8) -- (12);
\path[draw=black, -] (9) -- (10);
\path[draw=black, -] (10) -- (11);
\end{tikzpicture}
\hspace{30pt}
\begin{tikzpicture}
\path[draw=black] (1,3) -- (0,3) node[above]{$b$};
\path[draw=black] (1,3) -- (2,3) node[above]{$d$};
\path[draw=black] (1,3.5) -- (0,3.5) node[above]{$a$};
\path[draw=black] (1,3.5) -- (2,3.5) node[above]{$c$};
\path[draw=black] (1,3.5) -- (1,4)node[left]{$e$};
\node[circle, draw=black, fill=white] (H) at (1,3.25) {$H$};
\end{tikzpicture}
    \caption{An inverter component and diagram representation}
    \label{inverter}
\end{figure}
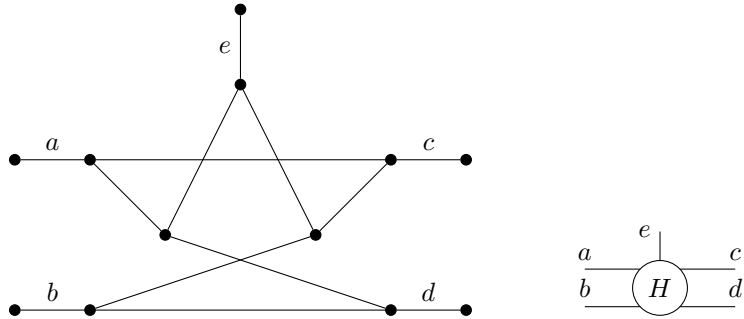

We will represent $H$ diagramatically as in the right image of Figure \ref{inverter}. In the following, we will refer to $a,b,c,d,$ and $e$ as \textbf{coding edges}. And, we will not include the degree 1 vertices when we refer to vertices of $H$. (The coding edges will connect to vertices in other components). 

\begin{lem}[Variable setting lemma] \label{lem: variable setting lemma}
For any $n$, there is a graph $V_n$ with $n$ pairs of distinguished edges $e_{i,0},e_{i,1}$ so that a 3 coloring $f$ of $e_{1,0},e_{1,1}...,e_{n,0},e_{n,1}$ extends to an edge 3-coloring of $V_n$ if and only if one of the following holds
\[(\forall i) \; f(e_{i,0})=f(e_{i,1})\]
or
\[(\forall i) \; f(e_{i,0})\not=f(e_{i,1})\]
\end{lem}
\begin{proof}
An example of $S_4$ is drawn in Figure \ref{variable}, with the $e_{i,j}$s being the eight edges connected to vertices of degree one. One can build $S_n$ for general $n$ by following a similar pattern (or chaining together copies of $S_4$.)
\begin{figure}[h]
    \centering
  \begin{tikzpicture}[scale=.8]
\tikzstyle{blnk}=[draw=none]
\tikzstyle{vtx}=[circle, draw=black, fill=black, inner sep=0pt, minimum size=4pt]
\node[vtx] (1) at (0,-.5) {};
\node[vtx] (2) at (.25,-.5) {};
\node[blnk] (3) at (0,.5) {};
\node[blnk] (4) at (.25,.5) {};
\node[blnk] (5) at (2.5,.5) {};
\node[blnk] (6) at (3.5,.5) {};
\node[vtx] (7) at (4.5,.5) {};
\node[blnk] (8) at (1.25,.75) {};
\node[blnk] (9) at (2.5,.75) {};
\node[blnk] (10) at (3.5,.75) {};
\node[vtx] (11) at (4.5,.75) {};
\node[blnk] (12) at (0,1.5) {};
\node[blnk] (13) at (.25,1.5) {};
\node[blnk] (14) at (1.75,1.5) {};
\node[blnk] (15) at (2.5,1.5) {};
\node[blnk] (16) at (3.25,1.5) {};
\node[blnk] (17) at (.25,2.5) {};
\node[blnk] (18) at (1,2.5) {};
\node[blnk] (19) at (2.25,2.75) {};
\node[blnk] (20) at (3.25,2.75) {};

\node[vtx] (22) at (-1,3.5) {};
\node[blnk] (23) at (0,3.5) {};
\node[blnk] (24) at (1,3.5) {};
\node[blnk] (25) at (2.25,3.5) {};
\node[vtx] (26) at (-1,3.75) {};
\node[blnk] (27) at (0,3.75) {};
\node[blnk] (29) at (3.25,3.75) {};
\node[blnk] (30) at (3.5,3.75) {};
\node[vtx] (31) at (3.25,4.75) {};
\node[vtx] (32) at (3.5,4.75) {};

\foreach \x/\y in {1/3, 2/4, 3/12, 4/5, 4/13, 5/6, 6/7, 9/10, 10/11, 10/30, 12/23, 20/29, 22/23, 23/24, 26/27, 27/29, 29/31, 30/32}{
\path[draw=black] (\x) -- (\y);}

\draw (9) .. controls (.75,.75) and (1.75,1.5) .. (13);
\draw (13) .. controls (.25,3) and (1,2) .. (24);
\draw (24) .. controls (2.75,3.5) and (1.75,2.75) .. (20);
\draw (20) .. controls (3.25,1) and (2.5,2) .. (9);

\node[circle, draw=black, fill=white, rotate=270] (H1) at (.125,.5) {$H$};
\node[circle, draw=black, fill=white, rotate=270] (H2) at (.125,1.5) {$H$};

\node[circle, draw=black, fill=white, rotate=180] (H3) at (0,3.625) {$H$};
\node[circle, draw=black, fill=white, rotate=180] (H4) at (1,3.625) {$H$};

\node[circle, draw=black, fill=white, rotate=90] (H3) at (3.375,2.75) {$H$};
\node[circle, draw=black, fill=white, rotate=90] (H4) at (3.375,3.75) {$H$};

\node[circle, draw=black, fill=white] (H3) at (2.5,.625) {$H$};
\node[circle, draw=black, fill=white] (H4) at (3.5,.625) {$H$};

\end{tikzpicture}
    \caption{A variable setting component}
    \label{variable}
\end{figure}
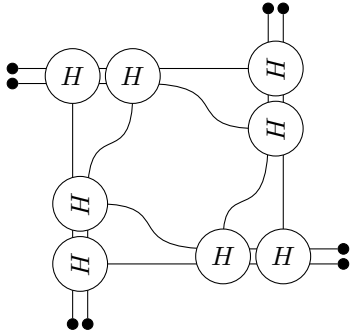
\end{proof}

As above, we will refer to the $e_{i,j}$s as coding edges and ignore the degrees one vertices in the following.

\begin{lem}[Or gate lemma] \label{lem: or gate}
There is a graph $O$ with three distinguished pairs of edges $e_{i,j}$ for $i=1,2,3$ and $j=0,1$ so that a 3 coloring $f$ of the $e_{i,j}$ extends to a coloring of $O$ if and only if $f(e_{i,0})=f(e_{i,1})$ for some $i$. 
\end{lem}
\begin{proof}
Such a graph is shown in Figure \ref{or}, where the $e_{i,j}$s are the 6 edges along the bottom of the image.

\end{proof}
\begin{figure}[h]
    \centering
    
\begin{tikzpicture}[scale=.8]
\tikzstyle{vtx}=[circle, draw=black, fill=black, inner sep=0pt, minimum size=4pt]
\node[vtx] (1) at (0,0) {};
\node[vtx] (2) at (.5,0) {};
\node[vtx] (3) at (2.5,0) {};
\node[vtx] (4) at (3,0) {};
\node[vtx] (5) at (5,0) {};
\node[vtx] (6) at (5.5,0) {};
\node[vtx] (7) at (.5,1) {};
\node[vtx] (8) at (1.5,1) {};
\node[vtx] (9) at (3,1) {};
\node[vtx] (10) at (4,1) {};
\node[vtx] (11) at (5.5,1) {};
\node[vtx] (12) at (6.5,1) {};
\node[vtx] (13) at (0,2) {};
\node[vtx] (14) at (.5,2) {};
\node[vtx] (15) at (2.5,2) {};
\node[vtx] (16) at (3,2) {};
\node[vtx] (17) at (5,2) {};
\node[vtx] (18) at (5.5,2) {};
\node[vtx] (19) at (2.75,4.75) {};

\draw[draw=black] (13) arc (180:0:2.75);

\foreach \x/\y in {1/13, 2/14, 3/15, 4/16, 5/17, 6/18, 7/8, 9/10, 11/12, 13/18}{
\path[draw=black] (\x) -- (\y);}

\node[circle, fill=white, draw=black, rotate=270] (H1) at (.25,1) {$H$};
\node[circle, fill=white, draw=black, rotate=270] (H1) at (2.75,1) {$H$};
\node[circle, fill=white, draw=black, rotate=270] (H1) at (5.25,1) {$H$};

\end{tikzpicture}
    \caption{An or gate}
    \label{or}
\end{figure}
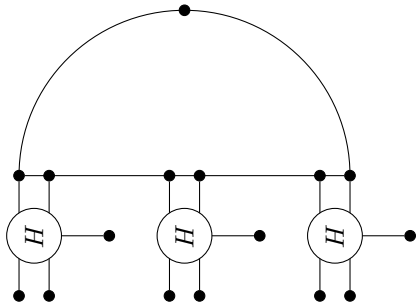

Once again, we will refer to the $e_{i,j}$s as coding edges and ignore the degree one vertices incident to coding edges.

Combining these lemmas, the reduction from 3SAT is straightforward

\begin{thm}\label{thm:edge coloring np complete}
There is a polynomial time reduction from 3SAT to 3 edge coloring
\end{thm}
\begin{proof}
Given an instance $\s X$ of 3SAT in CNF, build a graph $g(\s X)$ with one copy of $S_n$ for each variable which appears in $n$ disjunctions, one copy of $O$ for each disjunction of $\s X$. Wire each of the outgoing pairs of edge of each copy of $S_n$ into the appropriate copy of $O$, passing them through a copy of $H$ if the corresponding variable appears negated in the disjunction. See Figure \ref{reduction} for an example.

By construction, an edge 3-coloring of $g(\s X)$ yields a satisfying assignment of $\s X$ by assigning a variable True if the coloring agree on the outgoing pairs of edges in its copy of $S_n$ and False otherwise. And any satisfying assignment gives a partial coloring of the edges by the same scheme, which then extends an edge 3-coloring.

\end{proof}

\begin{figure}[h]
    \centering

\begin{tikzpicture}[scale=.45]

\tikzstyle{nvtx}=[circle, draw=black, fill=black, inner sep=0pt, minimum size=0pt]
\tikzstyle{vtx}=[circle, draw=black, fill=black, inner sep=0pt, minimum size=2pt]

\foreach \n/\x/\y in {1/0/0, 2/.5/0, 3/2/0, 4/3/0, 5/2/.5, 6/3/.5, 7/1/1, 8/1.5/1, 9/3.5/1, 10/4/1, 11/.5/1.5, 12/3/1.5, 13/3.5/1.5, 145/1.5/2, 14/2.5/2, 15/.5/2.5, 16/1/2.5, 17/3.5/2.5, 18/0/3, 19/.5/3, 20/2.5/3, 21/3/3, 22/1/3.5, 23/2/3.5, 24/1/4, 25/2/4, 26/3.5/4, 27/4/4}{
\node[nvtx] (\n) at (\x,\y) {};}

\foreach \x/\y in {1/18, 2/19, 3/4, 5/6, 8/145, 9/26, 10/27, 12/13, 14/20, 15/16, 22/23, 24/25}{\path[draw=black](\x)--(\y);} 

\foreach \x/\a/\b/\r in {11/90/0/.5, 3/270/180/1,5/270/180/.5, 12/270/180/.5, 145/0/90/.5, 20/0/90/.5, 17/270/180/.5, 21/0/90/1, 19/180/90/.5, 18/180/90/1, 6/270/360/.5, 4/270/360/1}{\draw[draw=black] (\x) arc (\a:\b:\r);}

\foreach \i/\x/\y in {1/.25/1.5, 2/.25/2.5}{
\node[circle, fill=white, draw=black, rotate=270] (H\i) at (\x,\y) {}; } 

\foreach \i/\x/\y in {3/3.75/1.5, 4/3.75/2.5}{
\node[circle, fill=white, draw=black, rotate=90] (H\i) at (\x,\y) {}; } 

\foreach \n/\x/\y in {1/0/0, 2/.5/0, 3/2/0, 4/3/0, 5/2/.5, 6/3/.5, 7/1/1, 8/1.5/1, 9/3.5/1, 10/4/1, 11/.5/1.5, 12/3/1.5, 13/3.5/1.5, 145/1.5/2, 14/2.5/2, 15/.5/2.5, 16/1/2.5, 17/3.5/2.5, 18/0/3, 19/.5/3, 20/2.5/3, 21/3/3, 22/1/3.5, 23/2/3.5, 24/1/4, 25/2/4, 26/3.5/4, 27/4/4}{
\node[nvtx] (\n) at (\x+5,\y) {};} 

\foreach \x/\y in {1/18, 2/19, 3/4, 5/6, 8/145, 9/26, 10/27, 12/13, 14/20, 15/16, 22/23, 24/25}{\path[draw=black](\x)--(\y);} 

\foreach \x/\a/\b/\r in {11/90/0/.5, 3/270/180/1,5/270/180/.5, 12/270/180/.5, 145/0/90/.5, 20/0/90/.5, 17/270/180/.5, 21/0/90/1, 19/180/90/.5, 18/180/90/1, 6/270/360/.5, 4/270/360/1}{\draw[draw=black] (\x) arc (\a:\b:\r);}

\foreach \i/\x/\y in {1/.25/1.5, 2/.25/2.5}{
\node[circle, fill=white, draw=black, rotate=270] (H\i) at (\x+5,\y) {}; } 

\foreach \i/\x/\y in {3/3.75/1.5, 4/3.75/2.5}{
\node[circle, fill=white, draw=black, rotate=90] (H\i) at (\x+5,\y) {}; } 

\foreach \n/\x/\y in {1/0/0, 2/.5/0, 3/2/0, 4/3/0, 5/2/.5, 6/3/.5, 7/1/1, 8/1.5/1, 9/3.5/1, 10/4/1, 11/.5/1.5, 12/3/1.5, 13/3.5/1.5, 145/1.5/2, 14/2.5/2, 15/.5/2.5, 16/1/2.5, 17/3.5/2.5, 18/0/3, 19/.5/3, 20/2.5/3, 21/3/3, 22/1/3.5, 23/2/3.5, 24/1/4, 25/2/4, 26/3.5/4, 27/4/4}{
\node[nvtx] (\n) at (\x+10,\y) {};} 

\foreach \x/\y in {1/18, 2/19, 3/4, 5/6, 8/145, 9/26, 10/27, 12/13, 14/20, 15/16, 22/23, 24/25}{\path[draw=black](\x)--(\y);} 

\foreach \x/\a/\b/\r in {11/90/0/.5, 3/270/180/1,5/270/180/.5, 12/270/180/.5, 145/0/90/.5, 20/0/90/.5, 17/270/180/.5, 21/0/90/1, 19/180/90/.5, 18/180/90/1, 6/270/360/.5, 4/270/360/1}{\draw[draw=black] (\x) arc (\a:\b:\r);}

\foreach \i/\x/\y in {1/.25/1.5, 2/.25/2.5}{
\node[circle, fill=white, draw=black, rotate=270] (H\i) at (\x+10,\y) {}; } 

\foreach \i/\x/\y in {3/3.75/1.5, 4/3.75/2.5}{
\node[circle, fill=white, draw=black, rotate=90] (H\i) at (\x+10,\y) {}; } 



\foreach \n/\x/\y in {1/0/0, 2/.5/0, 3/2.5/0, 4/3/0, 5/5/0, 6/5.5/0, 7/.5/1, 8/1.5/1, 9/3/1, 10/4/1, 11/5.5/1, 12/6.5/1, 13/0/2, 14/.5/2, 15/2.5/2, 16/3/2, 17/5/2, 18/5.5/2, 19/2.75/4.75}{
\node[nvtx] (\n) at (\x+6,\y+6) {};} 
\foreach \x in {8,10,12,13,14,15,16,17,18,19}{\filldraw[black] (\x) circle (2pt);}

\draw[draw=black] (13) arc (180:0:2.75);

\foreach \x/\y in {1/13, 2/14, 3/15, 4/16, 5/17, 6/18, 7/8, 9/10, 11/12, 13/18}{
\path[draw=black] (\x) -- (\y);}

\node[circle, fill=white, draw=black, rotate=270] (H1) at (6.25,7) {};
\node[circle, fill=white, draw=black, rotate=270] (H1) at (8.75,7) {};
\node[circle, fill=white, draw=black, rotate=270] (H1) at (11.25,7) {}; 


\foreach \n/\x/\y in {1/0/0, 2/.5/0, 3/2.5/0, 4/3/0, 5/5/0, 6/5.5/0, 7/.5/1, 8/1.5/1, 9/3/1, 10/4/1, 11/5.5/1, 12/6.5/1, 13/0/2, 14/.5/2, 15/2.5/2, 16/3/2, 17/5/2, 18/5.5/2, 19/2.75/4.75}{
\node[nvtx] (\n) at (\x+2.5,-\y-2) {};} 

\foreach \x in {8,10,12,13,14,15,16,17,18,19}{\filldraw[black] (\x) circle (2pt);}

\draw[draw=black] (13) arc (180:360:2.75);

\foreach \x/\y in {1/13, 2/14, 3/15, 4/16, 5/17, 6/18, 7/8, 9/10, 11/12, 13/18}{
\path[draw=black] (\x) -- (\y);}

\node[circle, fill=white, draw=black, rotate=270] (H1) at (2.75,-3) {};
\node[circle, fill=white, draw=black, rotate=270] (H1) at (5.25,-3) {};
\node[circle, fill=white, draw=black, rotate=270] (H1) at (7.75,-3) {}; 


\path[draw=black] (9,4) -- (9,6);
\path[draw=black] (8.5,4) -- (8.5,6);

\draw[black] (6,6) arc (0:-90:.5);
\draw[black] (6.5,6) arc (0:-90:1);

\path[draw=black] (4,4) -- (4,4.5);
\path[draw=black] (3.5,4) -- (3.5,4.5);

\draw[black] (4,4.5) arc (180:90:.5);
\draw[black] (3.5,4.5) arc (180:90:1);

\path[draw=black] (4.5, 5) -- (5.5,5);
\path[draw=black] (4.5, 5.5) --(5.5, 5.5);

\filldraw[black] (5,6.25) circle (2pt);
\path[draw=black] (5,6.25) --(5,5.25);
\node[circle, fill=white, draw=black] (H8) at (5,5.25) {};

\path[draw=black] (14,4) -- (14,4.5);
\path[draw=black] (13.5, 4) -- (13.5, 4.5);

\draw[black] (11.5,6) arc (180:270:.5);
\draw[black] (11,6) arc (180:270:1);

\draw[black] (14,4.5) arc (0:90:1);
\draw[black] (13.5,4.5) arc (0:90:.5);

\path[draw=black] (12, 5) -- (13,5);
\path[draw=black] (12, 5.5) --(13, 5.5);

\filldraw[black] (12.5,6.25) circle (2pt);
\path[draw=black] (12.5,6.25) --(12.5,5.25);
\node[circle, fill=white, draw=black] (H9) at (12.5,5.25) {};

\path[draw=black] (5,0) -- (5,-2);
\path[draw=black] (5.5,0) -- (5.5,-2);

\filldraw[black] (6.25,-1) circle (2pt);
\path[draw=black] (6.25,-1) -- (5.25,-1);
\node[circle, fill=white, draw=black, rotate=270] (H10) at (5.25, -1) {};

\draw[black] (2.5,-2) arc (0:90:.5);
\draw[black] (3,-2) arc (0:90:1);
\path[draw=black] (2,-1.5) -- (1,-1.5);
\path[draw=black] (2,-1) -- (1, -1);
\draw[black] (0,-.5) arc (180:270:1);
\draw[black] (.5,-.5) arc (180:270:.5);
\path[draw=black] (0,0) -- (0,-.5);
\path[draw=black] (.5,0) -- (.5,-.5);

\draw[black] (8,-2) arc (180:90:.5);
\draw[black] (7.5,-2) arc (180:90:1);

\path[draw=black] (8.5,-1.5) -- (9.5, -1.5);
\path[draw=black] (8.5, -1) -- (9.5, -1);

\draw[black] (9.5,-1) arc (-90:0:.5);
\draw[black] (9.5,-1.5) arc (-90:0:1);

\path[draw=black] (10, -.5) -- (10,0);
\path[draw=black] (10.5,-.5)-- (10.5,0);
\end{tikzpicture}
    \caption{The reduction applied to $(\neg v_1\vee v_2 \vee \neg v_3)\wedge (v_1\vee \neg v_2 \vee v_3)$}
    \label{reduction}
\end{figure}


We will verify that the same construction can be carried out for locally finite instances of Borel 3SAT. First, we fix a coding for Borel sets. The important point is Lemma \ref{coding lemmas}, which lets us easily check that a construction is $\Delta^1_1$ in the codes.

\begin{thm}\label{thm: good parameterization}
There is a good $\omega$-parameterization of $\Pi^1_1$, i.e. a $\Pi^1_1$ set $U\subseteq \omega\times \baire$ so that
\end{thm}
\begin{enumerate}
    \item For every $\Pi^1_1$ $P\subseteq \baire$ there is an $e$ so that $P=U_e=\{x\in \baire: (e,x)\in U\}$.
    \item For every $\Pi^1_1$ $P\subseteq \omega^{k}\times \baire$ there is a recursive function $S: \omega^k\rightarrow \omega$ so that, for $n\in\omega^k$,
    \[(n, x)\in P\lra (S(n),x)\in U\]
\end{enumerate}

Fix such a good $\omega$-parameterization $U$.

\begin{dfn}\label{dfn:nice codes}
Fix a $\Delta^1_1$ set $B$. A \textbf{simple codes} for $B$ is a pair $\ip{e,i}\in\omega^2$ so that
\[U_e=\baire\sm U_i=B\] for some good $U$.

A \textbf{nice coding} is a triple $(\s C, \s D^{\Pi}, \s D^{\Sigma})$ where

\begin{enumerate}
    \item $\s C\subseteq \omega$ is a $\Pi^1_1$ set, referred to as the codes
    \item $\s D^{\Pi}, \s D^{\Sigma}\subseteq \omega \times \baire$, $\s D^{\Pi}$ is $\Pi^1_1$, and $\s D^{\Sigma}$ is $\Sigma^1_1$
    \item For every $e\in \s C$, $\s D_e^{\Pi}=\s D_e^{\Sigma}$
    \item For every $\Delta^1_1$ set $B\subseteq\baire$ there is a code $e\in \s C$ with $B=D^{\pi}_e=D^{\Sigma}_e$
    \item There are recursive functions $f,g$ so that, for every $B$, 
    \[e\mbox{ is a nice code for }\rightarrow f(e)\mbox{ is a simple code for }B\]
    \[\ip{e,i}\mbox{ is a simple code for }B\rightarrow g(e,i)\mbox{ is a nice code for }B\]
\end{enumerate}
\end{dfn}

One can obtain nice codes from the simple codes by a uniform application of separation. See \cite[Section 3.3]{moschovakis}

We'll fix simple and nice codings for $\Delta^1_1(\baire^k)$ for all $k$, and for $e\in \s C$ write $\s D_e$ for $\s D^{\Pi}_e$.

\begin{lem}\label{coding lemmas} 
Fix a $\Delta^1_1$ linear order $\preceq$ of $\baire$. The following maps are $\Delta^1_1$ in the codes:
\begin{enumerate}
    \item $(A,B)\mapsto A\cup B$
    \item $(A,B)\mapsto A\cap B$
    \item $(A,B)\mapsto A\times B$
    \item $A\mapsto \dom(A)$, where $A$ is a relation with countable sections.
    \item $(A,f)\mapsto f(A)$, where $f$ is a countable-to-one function
    \item $A\mapsto f$ where $A$ is a relation with finite sections, and 
    \[f(x,i)=y\lra y\mbox{ is the $i^{th}$ element of $R_x$ according to }\preceq\] 
\end{enumerate}
That is, there is a function $f:\omega^2\rightarrow \omega$ so that, if $e,i$ are nice codes for $A$ and $B$, then $f(e,i)$ is a nice code for $A\cup B$ (and likewise for the other items).
\end{lem}
\begin{proof}
We'll prove $(1)$ and $(4)$ and sketch $(6)$, the other items are easy once you've seen the general idea.

Using item $(5)$ of the definition of nice codes, it suffices to show that, for our good $\omega$-parameterization $U$, there are $\Delta^1_1$ functions $P,S$ so that, for any $e,i,j,k\in\omega$, if 
\[A=U_e=\baire \sm U_i\quad B=U_j=\baire\sm U_k\]
then
\[U_{P(e,i,j,k)}=\baire\sm U_{S(e,i,j,k)}=A\cup B.\]

Define \[R_P(e,i,j,k,x):\lra (e,x)\in U\mbox{ or }(j,x)\in U\]
\[R_S(e,i,j,k,x):\lra (i,x)\in \omega\times \baire\sm U\mbox{ or }(k,x)\in \omega\times \baire\sm U.\] Then $R_P$ is $\Pi^1_1$ and $R_S$ is $\Sigma^1_1$. By item $(2)$ of the definition of good $\omega$-parameterizations, there are recursive functions $P,S$ so that, for any $e,i,j,k\in\omega$ and $x\in\baire$
\[R_P(e,i,j,k,x)\lra P(e,i,j,k)\in U\]
\[R_S(e,i,j,k,x)\lra S(e,i,j,k)\not\in U\] as desired.

Similarly, for item $(4)$ we want $\Delta^1_1$ functions $S,P$ so that, if $\ip{e,i}$ is a simple code for a relation $A\subseteq\baire^2$ with countable sections, then $\ip{S(e,i), P(e,i)}$ is a simple code for the domain of $A$. Define

\[R_P(e,i,x)\lra (\exists y\in\Delta^1_1(x))\;(e,x,y) U_e\]
\[R_S(e,i,x)\lra (\exists y)\;(e,x,y)\in U_e\]

Note that $R_P$ is $\Pi^1_1$ and $R_S$ is $\Sigma^1_1$. By the effective perfect set theorem, if $A$ is relation with countable sections, and $\ip{e,i}$ is a simple code for $A$, then for any $x$, $R_P(e,i,x)\lra R_S(e,i,x)\lra x\in \dom(A).$ So, again we can find recursive functions as desired.

For $(6)$, we use the following definition of $f$: $f(x,i)=y$ if and only if
\[(\exists y_1,...,y_i\in R_x)\left[y_1\prec ... \prec y_i=y\right]\] and \[ (\forall y_1,...,y_{i+1}\in R_x)\left[ y_1\prec ...\prec y_{i+1}\Rightarrow y_{i+1}\not=y \right].\] And, either quantifier can be taken to range over $\Delta^1_1(x).$
\end{proof}


This lemma can be usually be used a black box without the need to delve into the details of the coding.

\begin{dfn}\label{dfn:edge colorings}
$\mathbf{E}$ is the set of (nice) codes for Borel graphs with Borel edge 3-colorings, where we view a graph as a vertex set $V\subseteq\baire$ and a symmetric subset of edges $E\subseteq V^2$. That is $\mathbf{E}$ is
\[\{\ip{v,e}\in \s C^2: D_e\subseteq D_{v}^2\mbox{ and }D_e\mbox{ is symmetric and Borel edge 3-colorable}\}\]
$\csp_B^{lf}(3\sat)$ is the set of codes for locally finite Borel instances of 3SAT with Borel satisfying assignments.
\end{dfn} 

It follows from the comments following \ref{thm:intractable} that $\csp_B^{lf}(3\sat)$ is $\bs12$-complete.

\begin{thm}\label{thm: edge coloring borel complete}
$\csp_B^{lf}(3\sat)\leq_B \mathbf{E}$.
\end{thm}
\begin{proof}
The construction in Theorem \ref{thm:edge coloring np complete} works in the Borel setting. We will describe this construction formally so that it is clear that $G_{\s X}$ is generated from $\s X$ using operations from Lemma \ref{coding lemmas}, and then use the Luzin--Novikov theorem to verify that this gives a reduction.

Fix an locally finite Borel instance $\s X$ of 3SAT in CNF with variables $V$ and constraints $C=\{c_1,...,c_m\}$. We may assume no variable shows up twice in any $c_i$. We define the following parameters:
\begin{itemize}
\item for a variable $v$, $r(v)$ is the number of constraints $v$ or $\neg v$ appears in
\item for any variable $v$, $c^v_1,...,c^v_{r(v)}$ lists the constraints $v$ or $\neg v$ appears in
\item for any constraint $c$, $v^c_1,v^c_2,v^c_3$ lists the variables which appear in $c$
\item $N=\{(v,c)\in V\times C: \neg v\mbox{ appears in }c\}$
\end{itemize}

The vertex set of $G_x$ is \[V_{\s X}:=\bigcup_{v\in V} \{v\}\times S_{r(i)}\cup \bigcup_{c\in C} \{c\}\times O \cup \bigcup_{(v,c)\in N} \{(v,c)\}\times H.\] Note that $V_{\s X}$ can built from $\s X$ and the components of $O, H,$ and $S_{r(n)}$ using products, intersections, and unions. Since $O, H,$ and $S_{r(n)}$ all have computable codes, $\s X\mapsto V_{\s X}$ is $\Delta^1_1$ in the codes. 

The edge set of $G_{\s X}$, $E_{\s X}$, includes the following edges:
\begin{enumerate}
    \item Any $\{x\}\times e\in V_{\s X}^2$ with $e$ a non-coding edge
    \item $\{(v, u),(c,w)\}$ where $v$ appears in $c$ and $u,w$ are on corresponding coding edges, i.e. where $c=c^{v}_i$, $v=v^c_j$, $(v,c)\not\in N$, and $u$ is on $e_{i,k}$ in $S_{r(v)}$ and $w$ is on $e_{j,k}$ in $O$ for some $k\in \{0,1\}$.
    
    \item $\{(v,u),(v,c,w)\}$ where $(v,c)\in N$, $c=c_i^v$, and either $u$ is on $e_{i,0}$ in $S_{r(v)}$ and $w$ is on $a$ in $H$ or $u$ is on $e_{i,1}$ and $w$ is on $b$.
    \item $\{(v,c,w),(c,u)\}$ where $v=v_j^c$, $(v,c)\in N$, and either $u$ is on $e_{j,0}$ in $O$ and $w$ is on $c$ in $H$ or $u$ is on $e_{j,1}$ and $w$ is on $d$.
\end{enumerate}

Again, $E_{\s X}$ can built from codes for $\s X$, $O$, $H$, and $S_{r(n)}$ using operations from Lemma \ref{coding lemmas}, so $\s X\mapsto G_{\s X}=(V_{\s X}, E_{\s X})$ is $\Delta^1_1$ in the codes.

As before, any Borel edge 3 coloring of $\s G_x$ induces a Borel satisfying assignment of $\s X$ by setting a variable to True if the coloring agrees on corresponding pairs of edges and False otherwise. 

For the converse, recall the Luzin--Novikov theorem which says that Borel assignment of countable sets to points in a standard Borel space admits a selection function. Given a Borel satisfying assignment of $\s X$, there is a Borel partial coloring $f$ of the corresponding edges in each copy of $S_n$ which extends to a not-necessarily Borel edge 3-coloring. On each copy of $S_n$, $H$, and $O$, there is are finitely many edge colorings consistent with $f$. Using the Lusin--Novikov theorem, we can select a such a coloring on each of these components and get a Borel edge 3-coloring of $ G_{\s X}$. 
\end{proof}

\end{document}